\documentclass[a4paper]{amsart}
\input xypic
\input xy \xyoption{all}
\usepackage{bbm}
\usepackage{amssymb,graphicx,float,fancyhdr,tikz,dsfont,enumerate,mathtools,booktabs,blkarray,multirow}
\usepackage{amsmath}
\usepackage{graphicx}
\usetikzlibrary{matrix}

\makeindex

\newtheorem{theorem}{Theorem}[section]
\newtheorem{proposition}[theorem]{Proposition}
\newtheorem{lemma}[theorem]{Lemma}

\theoremstyle{definition}

\newtheorem{rmk}[theorem]{Remark}

\def\multiset#1#2{\ensuremath{\left(\kern-.3em\left(\genfrac{}{}{0pt}{}{#1}{#2}\right)\kern-.3em\right)}}
\def\supermultiset#1#2{\ensuremath{\:\left(\kern-.5em\left(\genfrac{}{}{0pt}{}{#1}{#2}\right)\kern-.5em\right)\:}}

\DeclareMathOperator\Map{\mathrm{Map}}

\begin{document}
\title{The cohomology of free loop spaces of rank 2 flag manifolds}

\author{Matthew I. Burfitt}
\address{\scriptsize{School of Mathematics, University of Southampton,Southampton SO17 1BJ, United Kingdom}}  
\email{M.Burfitt@soton.ac.uk}

\author{Jelena Grbi\'c} 
\address{\scriptsize{School of Mathematics, University of Southampton,Southampton SO17 1BJ, United Kingdom}}  
\email{J.Grbic@soton.ac.uk} 

\subjclass[2010]{}
\thanks{Research supported in part by The Leverhulme Trust Research Project Grant  RPG-2012-560.}

\begin{abstract}
A complete flag manifold is the quotient of a Lie group by its maximal torus. The rank of a flag manifold is the dimension of the maximal torus of the Lie group. The rank 2 complete flag manifolds are $SU(3)/T^2$, $Sp(2)/T^2$, $Spin(4)/T^2$, $Spin(5)/T^2$ and $G_2/T^2$. In this paper we calculate the cohomology of the free loop space of rank 2 complete flag manifolds.  \end{abstract}

\maketitle

\section{Introduction}

One of the most influential problems in topology and geometry is the study of geodesics on Riemannian manifolds. The geodesic problem refers to finding geodesics connecting two given points of a Riemannian manifold or to finding periodic geodesics, and to giving information regarding their count. The study of geodesics originated with the works of Hadamard~\cite{Hadamard1898} and Poincar\' e~\cite{MR1500710} and with substantial early contributions by Birkhoff~\cite{MR1501070}, Morse~\cite{MR1451874}, and Lusternik and Schnirelman~\cite{MR0029532}. The most important offspring of this problem is the development of topological methods in variational calculus, generally referred to as Morse theory~\cite{MR1001450}. Recently Floer theory developed as a central tool in modern symplectic topology taking its inspiration from the study of geodesics. The geodesic problem also led to the development of computational tools in algebraic topology (spectral sequences), and is connected to the theory of minimal models and to Hochschild and cyclic homology.

One of the most natural starting points in the study of the geodesic problem is the study of spaces of paths and loops on manifolds. In recent years, these spaces have been the object of much interest in topology, symplectic geometry and theoretical physics.  The \textit{free loop space} $\Lambda X$ of a topological space $X$ is defined to be the mapping space $\Map(S^1,X)$, the space of all non-pointed maps from the circle to $X$.

Given a Riemannian manifold $(M,g)$, the closed geodesics parametrised by $S^1$ are the critical points of the energy functional
\[
E\colon \Lambda M\to \mathbb R, \quad E(\gamma):= \frac{1}{2}\int_{S^1} ||\dot{\gamma}(t)||^2 dt.
\]
Morse theory applied to the energy functional $E$ gives a description of the loop space $\Lambda M$ by successive attachments of bundles over the critical submanifolds with rank given by the index of the Hessian $d^2E$. This allows a grip on the topology of $\Lambda M$ provided one has enough information on these indices and on the attaching maps. Conversely, knowledge of the topology of $\Lambda M$ implies existence results for critical points of $E$.

One of the most powerful results in this direction is due to Gromoll-Meyer~\cite{MR0264551} who proved that when $M$ is a simply connected closed manifold such that the sequence $\{b_k(\Lambda M)\}, k\geq 0$ of Betti numbers of  $\Lambda M$ with coefficients in some field is unbounded, then for any Riemannian metric on $M$ there exist infinitely many geometrically distinct closed geodesics.
 
 A distinctive subspace of $\Lambda M$ is the \textit{based loops space} $\Omega X=\Map_*(S^1,X)$, the space of all pointed maps from the circle to $X$. The based loop space functor is an important classical object in algebraic topology and has been well studied. The topology of free loop spaces is much less well behaved and it is still only well understood in a handful of examples. In particular, the cohomology of free loop spaces such as spheres, $n$-dimensional projective spaces and Lie groups.

The starting point for the topological study of the free loop space  $\Lambda M$ is the evaluation fibration 
 \[
\Omega M\to \Lambda M\stackrel{ev}{\to} M
\]
where  $ev$ is the evaluation at the origin of a loop, and  $\Omega M$ is the based loop space, consisting of loops starting and ending at a fixed basepoint in $M$. This fibration can be used to determine the homotopy groups of  $\Lambda M$, that is, $\pi_k(\Lambda M)\cong \pi_k(M)\oplus\pi_k(\Omega M)$: the section given by the inclusion of constant loops determines a splitting of the homotopy long exact sequence. However, the situation is very different as far as homology groups are concerned. It turns out that the Leray-Serre spectral sequence is effective in simple cases (spheres~\cite{MR1048178, Dinh}) but of very limited use in general, unless one has additional geometric information about the differentials. In contrast to the evaluation fibration, the path-loop fibration $\Omega M\to PM\to M$ has been successfully used to study $\Omega M$ due to it having the contractible total space $PM$. For example, the author and Terzi\' c~\cite{MR2722368, MR3076855} calculate the integral Pontryagin homology ring of flag manifolds and of generalised  symmetric spaces. Any successful reasoning in the study of free loop spaces must use specific features of the evaluation fibration.

In this paper we explore the cohomology of the free loop space of homogeneous spaces. In doing so we uncover some surprising combinatorial connections and explicitly compute the cohomology of the free loop space of flag manifolds of rank 2 simple Lie groups.

A compact connected Lie group is called simple if it is non-abelian, simply connected and has no non-trivial connected normal subgroups. The only compact connected simple Lie groups are $Spin(m)$, $SU(n)$, $Sp(n)$, $G_2$, $F_4$, $E_6$, $E_7$, $E_8$ for $n\geq 1$ and $m\geq 2$. In this paper we specialise to rank 2 simple Lie groups $SU(3)$, $Sp(2)$, $Spin(4)$, $Spin(5)$ and $G_2$.
In low dimensions, there are isomorphisms among the classical Lie groups called accidental isomorphisms, identifying certain simple Lie groups of rank 2 such as $Spin(4)\cong SU(2)\times SU(2)$ and $Spin(5)\cong Sp(2)$. 
Therefore, in this paper we focus on the cohomology of the free loop spaces on $SU(3)/T^2$, $Sp(2)/T^2$ and $G_2/T^2$.

\section{Background}

		\subsection{Gr\"{o}bner bases}		
		
		Gr\"{o}bner basis provide us with powerful methods to perform computations in commutative algebra
		particularly with respect to ideals, although their application extend far beyond such calculations.
		In this subsection we briefly describe the Gr\"{o}bner basis theory to be used later in this paper,
		for more information and proofs see \cite{Grobner2} or \cite{Grobner1}.
		We state all results over Euclidean or principle ideal domain $R$;
		in the paper we will only consider the case $R=\mathbb{Z}$ for which all results hold.
		The theory of Gr\"{o}bner basis can be generalised to other rings and stronger results can be recovered over fields.
		
		Let $R$ be a ring. Given a finite subset $A$ of $R[x_1,\dots,x_n]$, we denote by $\langle A\rangle$ the ideal generated by elements of $A$.
		Fix a monomial ordering on the polynomial ring $R[x_1,\dots,x_n]$.
		For $f,g,p\in R[x_1,\dots,x_n]$, $g$ is said to be {\it reduced} from $f$ by $p$
		if there exists a term $m$ in $f$ such that the leading term of $p$ divides $m$ and $g=f-m'p$
		for some monomial $m'\in R[x_1,\dots,x_n]$.
		
		Let $R$ be a principle ideal domain and let $G$ be a finite subset of $R[x_1,\dots,x_n]$.
		Then $G$ is a {\it Gr\"{o}bner basis} if all elements of $\langle G\rangle$ can be reduced to zero by elements of $G$.
		
		A ring $R$ is called {\it computable} if for any pairs of ring elements there exists both
		a division algorithm and a euclidean algorithm. 
		
		\begin{theorem}[\cite{Grobner1}]
		    Let $R$ be a computable principle ideal domain and
		    fix a monomial order on $R[x_1,\dots,x_n]$.
		    For any ideal in $R[x_1,\dots,x_n]$, there exists a Gr\"{o}bner basis.
		    In particular, for finite $A\subseteq R[x_1,\dots,x_n]$ there is an algorithm
		    to obtain Gr\"{o}bner basis $G$ such that $\langle G\rangle=\langle A\rangle$. 
		\end{theorem}
		
		The most efficient algorithm is known as a the Buchberger algorithm and can easily be implemented by a computer.
		
		\begin{theorem}[\cite{Grobner1}]\label{thm:GrobnerOver}
		    Let $R$ be a Euclidean domain with unique remainders and let $G$ be a Gr\"{o}bner basis in $R[x_1,\dots,x_n]$.
		    Then all elements in $R[x_1,\dots,x_n]$ reduce to a unique representative in $R[x_1,\dots,x_n]/\langle G\rangle$.
		\end{theorem}
		
		In particular, Gr\"{o}bner basis can be used to compute the intersection of ideals
		which we make explicit in the next remark.
		
		\begin{rmk}
		    Let $A=\{a_1,\dots,a_s\}$ and $B=\{b_1,\dots,b_l\}$ be subsets of $R[x_1,\dots,x_n]$.
		    Take a Gr\"{o}bner basis $G$ of 
		    \begin{equation*}
		        \{ya_1,\dots,ya_t,(1-y)b_1,\dots,(1-y)b_l\}
		    \end{equation*}
		    in $R[x_1,\dots,x_n,y]$ using a monomial ordering in which monomials containing $y$ are larger than $y$ free monomials.
		    Then a Gr\"{o}bner basis of $\langle A\rangle\cap\langle B\rangle$ is given by the elements of $G$ that do not contain $y$.
		\end{rmk}

        \subsection{Cohomology of complete flag manifolds of simple Lie groups}
        		
        A Lie subgroup $T$ of a Lie group $G$ isomorphic to a torus is called maximal if any subtorus containing $T$ coincides with $T$.
        Maximal tori are conjugate and cover the Lie group, therefore it is unambiguous to refer to the maximal torus $T$ of $G$.
        The homogeneous space $G/T$, isomorphic regardless of the choice of $T$, is called the {\it complete flag manifold} of $G$.
        The rank of Lie group $G$ is the dimension of a maximal torus~$T$.
        		
        The cohomology of homogeneous spaces was studied in detail by Borel in \cite{Borel}. In particular, from Borel's work it is possible to deduce the rational cohomology of $G/T$.
        \begin{theorem}[\cite{Borel}]\label{thm:Borel}
        For compact connected Lie group $G$ with maximal torus $T$,
        \begin{equation*}
        H^*(G/T;\mathbb{Q})\cong \frac{H^*(BT;\mathbb{Q})}{\tilde{H}^*(BT;\mathbb{Q})^{W_G}}
        \end{equation*}
        where $BT$ is the classifying space of $T$ and $W_G$ is the Weyl group of Lie group $G$.
        \end{theorem}
        In \cite{torG/T} Bott and Samelson, using Morse theory, extended Borel's work by showing that there is no torsion in $H^*(G/T;\mathbb{Z})$.
        This made it easier to deduced the integral structure of the cohomology of complete flag manifolds in the cases of $SU(n),Sp(n)$ and $G_2$.
        Toda in \cite{toda1975} studied the cohomology of homogeneous spaces looking at the mod $p$ cohomology for prime $p$.
        In particular, Toda described in a nice form the integral cohomology algebras of complete flag manifolds in the case of $SO(n)$.
        Soon after, Toda and Watanabe~\cite{toda1974} computed the integral cohomology of the complete flag manifolds of $F_4$ and $E_6$.
        The cohomology of complete flag manifolds of simple Lie groups was completed by Nakagawa who described the cases for $E_7$ and $E_8$ in \cite{NakagawaE7} and \cite{nakagawaE8}, respectively. 
        
        We recall the cohomology rings of flag manifolds used in this paper following
        \cite{Borel} and \cite{AplicationsOfMorse}.
        The cohomology of the complete flag manifold of the simple Lie group $SU(3)$ is given by
        \begin{equation*}
        H^*(SU(3)/T^2;\mathbb{Z})
        =\frac{\mathbb{Z}[\gamma_1,\gamma_2,\gamma_3]}{\langle\sigma_1,\sigma_2,\sigma_3\rangle}
        \end{equation*}
        where $|\gamma_i|=2$ for $i=1,2,3$ and $\sigma_i$ are the elementary symmetric polynomials of degree $i$ in variables $\gamma_1,\gamma_2,\gamma_3$.
        To simplify calculations, using $\sigma_1=\gamma_1+\gamma_2+\gamma_3$ rewrite variable $\gamma_3$
        as $\gamma_3=-\gamma_1-\gamma_2$, hence
        \begin{equation}\label{eq:H*SU/T}
        H^*(SU(3)/T^2;\mathbb{Z})=\frac{\mathbb{Z}[\gamma_1,\gamma_2]}{\langle\sigma_2,\sigma_3\rangle}.
        \end{equation}
        Note that we also rewrite $\sigma_2=\gamma_1^2+\gamma_2^2+\gamma_1\gamma_2$ and
        $\sigma_3=\gamma_1^2\gamma_2+\gamma_1\gamma_2^2$.
        
        The cohomology of the complete flag manifold of the simple Lie group $Sp(2)$ is given by
        \begin{equation}\label{eq:H*Sp/T}
        H^*(Sp(2)/T^2;\mathbb{Z})=\frac{\mathbb{Z}[\gamma_1,\gamma_2]}{\langle\sigma_1^2,\sigma_{2}^2\rangle}
        \end{equation}
        where $|\gamma_i|=2$ for $i=1,2$ and $\sigma_i^2$ denotes elementary symmetric polynomial of degree $i$ in variables $\gamma_1^2,\gamma_2^2$.
        
        The cohomology of the complete flag manifold of the exceptional simple Lie group $G_2$ is given by
        \begin{equation*}
        H^*(G_2/T^2;\mathbb{Z})=\frac{\mathbb{Z}[\gamma_1,\gamma_2,\gamma_3,t_3]}{\langle\sigma_1,\sigma_2,\sigma_3-2t_3,t_3^2\rangle}
        \end{equation*}
        where $|\gamma_i|=2$ for $i=1,2,3$, $|t_3|=6$
        and $\sigma_i$ denotes the elementary symmetric polynomial of degree $i$ in variables $\gamma_1,\gamma_2,\gamma_3$.
        Again to simplify calculations, using $\sigma_1=\gamma_1+\gamma_2+\gamma_3$ we rewrite varaiable $\gamma_3$  as $\gamma_3=-\gamma_1-\gamma_2$, hence
        \begin{equation}\label{eq:H*G2/T}
        	H^*(G_2/T^2;\mathbb{Z})
        	=\frac{\mathbb{Z}[\gamma_1,\gamma_2,t_3]}{\langle\sigma_2,\sigma_3-2t_3,t_3^2\rangle}.
        \end{equation}
		Similarly, rewrite $\sigma_2=\gamma_1^2+\gamma_2^2+\gamma_1\gamma_2$ and
        $\sigma_3=\gamma_1^2\gamma_2+\gamma_1\gamma_2^2$.

		\subsection{Based loop space cohomology of simple Lie groups}
		
		The Hopf algebra of the based loop space of Lie groups were studied by Bott \cite{bott1958}.
		We recall the results used later in the paper.
	Recall the integral divided polynomial algebra on variables $x_1,\dots,x_n$ is given by
		\begin{equation*}
			\Gamma_{\mathbb{Z}}[x_1,\dots,x_n]
			=\frac{\mathbb{Z}[(x_i)_1,(x_i)_2,\dots]}{\langle(x_i)_k-k!x_i^k\rangle}
		\end{equation*}
		for $1\leq i \leq n$, $k\geq 1$ where $(x_i)_1=x_i$.
		The following results can be obtained from the cohomology of $SU(3)$ and $Sp(2)$
		by applying the Leray-Serre spectral sequence to the path-loop fibrations
		\begin{equation*}
		    \Omega SU(3) \to PSU(3) \to SU(3) \quad
		    \text{and}\quad
		    \Omega Sp(2) \to PSp(2) \to Sp(2).
		\end{equation*}
		The integral cohomology of the based loop space of the classical simple Lie group $SU(3)$ is given by
		\begin{equation}\label{eq:LoopSU}
			H^*(\Omega(SU(3));\mathbb{Z})=\Gamma_{\mathbb{Z}}[x_2,x_4]
		\end{equation}
		where $|x_2|=2$ and $|x_4|=4$.
		The integral cohomology of the based loop space of the classical simple Lie group $Sp(2)$ is given by
		\begin{equation}\label{eq:LoopSp}
			H^*(\Omega(Sp(2));\mathbb{Z})=\Gamma_{\mathbb{Z}}[x_2,x_6]
		\end{equation}
		where $|x_2|=2$ and $|x_6|=6$.
		It is less straightforward to, in a similar way, compute the integral cohomology of $\Omega G_2$.
		
		\begin{proposition}\label{prop:LoopG2}
	    The integral cohomology of the based loop space of the exceptional simple Lie group $G_2$ is given by
		\begin{equation*}
			H^*(\Omega G_2;\mathbb{Z})=\frac{\mathbb{Z}[(a_2)_1,(a_2)_2,\dots]}{\langle a_2^{m}-(m!/2^{\lfloor\frac{m}{2}\rfloor})(a_2)_{m}\rangle}\otimes \Gamma_{\mathbb{Z}}[b_{10}]
		\end{equation*}
		where $m\geq 1$, $(a_2)_1=a_2$ and $|(a_2)_m|=2m$, $|b_{10}|=10$.
	    \end{proposition}
        
        \begin{proof}
		    The integral cohomology of $G_2$~\cite{CohomologyLieGroup} is give by
		    \begin{equation*}
		        H^*(G_2;\mathbb{Z})=\frac{\mathbb{Z}[x_3,x_{11}]}{\langle x_3^4, \; x_{11}^2, \; 2x_3^2, \; x_3^2x_{11} \rangle}
		    \end{equation*}
		    where $|x_3|=3$ and $|x_{11}|=11$.
		    Since $G_2$ is simply connected,
		    consider the Leray-Serre spectral sequence $\{ E_r, d^r \}$ of the path-loop fibration
		    \begin{equation*}
		        \Omega G_2 \to P G_2 \to G_2.
		    \end{equation*}
		    Spectral sequence $\{ E_r, d^r \}$ converges to the trivial algebra and the first non-trivial class in the cohomology of $G_2$ is generated by $x_3$ in degree $3$.
		    Hence, for dimensional reasons there is $a_2\in H^2(\Omega G_2;\mathbb{Z})$ and 
		    \begin{equation*}
		        d^3(a_2)=x_3.
		    \end{equation*}
		    For $m\geq 1$, denote by $(a_2)_m$ the additive generator of $H^{2m}(\Omega G_2;\mathbb{Z})$
		    such that $c(a_2)_m=a_2^m$ for some $c\geq 0$.
		    Element $x_3^2$ is $2$-torsion and $d^3(a_2x_3)=x_3^2$,
		    so the kernel of $d^3 \colon E^3_{3,2} \to E^3_{6,0}$ is generated by $2a_2x_3$.
		    Hence, as $d^3(a_2^2)=2a_2x_3$,
		    \begin{equation*}
		        (a_2)_2=a_2^2.
		    \end{equation*}
		    As the element $a_2x_3^2$ is 2-torsion and $d^3((a_2)_2x_3)=2a_2x_3^2=0$,
		    the kennel of $d^3 \colon E^3_{3,4} \to E^3_{6,2}$ is generated by $(a_2)_2x_3$.
		    Hence, as $d^3(a_2^3)=3(a_2)_2x_3$,
		    \begin{equation*}
		        3(a_2)_3=a_2^3.
		    \end{equation*}
		    For $p \geq 1$, $q\geq 2$ even and $r\geq 1$ odd,
		    \begin{equation*}
		        \ker(d^3)(E_3^{p,2q})\cong H^*(G_2,\mathbb{Z})
		        \text{ and }
		        \ker(d^3)(E_3^{p,2r})\cong E_3^{p,2}.
		    \end{equation*}
		    In particular, it can be shown inductively that
		    $\{ a_2=(a_2)_1,(a_2)_2,(a_2)_3,\dots \}$ generate an algebra
		    \begin{equation*}
		        \frac{\mathbb{Z}[(a_2)_1,(a_2)_2,\dots]}{\langle a_2^{m}-(m!/2^{\lfloor\frac{m}{2}\rfloor})(a_2)_{m} \rangle}.
		    \end{equation*}
		    As element $(a_2)_{q-1}x_3^2$ is 2-torsion and $d^3((a_2)_qx_3)=(a_2)_{q-1}x_3^2$,
		    the kernel of $d^3 \colon E^3_{3,2q} \to E^3_{6,2q-2}$ is generated by $2(a_2)_{q-1}x_3$.
		    As $d^3(a_2^q)=qa_2^{q-1}x_3$, by induction,
		    \begin{equation*}
		        q!/2^q(a_2)_q=a_2^q.
		    \end{equation*}
		    As element ${a_2}_rx_3^2$ is 2-torsion and $d^3((a_2)_{r+1}x_3)=a_rx_3^2$,
		    the kennel $d^3 \colon E^3_{3,2r} \to E^3_{6,2r-2}$ is generated by $(a_2)_{r-1}x_3$.
		    Hence, as $d^3(a_2^r)=r(a_2)^{r-1}x_3$, by induction
		    \begin{equation*}
		        r!/2^{r-1}(a_2)_r=a_2^r.
		    \end{equation*}
		    For dimensional reasons, there are no other non-trivial $d^3$ differentials other than those occurring on multiple of $(a_3)_m$.
		    Hence the only non-trivial elements on the $E_4$-page divisible by $(a_2)_m, \; x_3$ or $x_{11}$
		    for $m \geq 1$ are generated by
		    \begin{equation*}
		        x_{11}, \; (a_2)_tx_2, \; (a_2)_tx_3x_{11}
		    \end{equation*}
		    where $t\geq 1$ is odd,
		    $x_{11}$ is non-torsion, $(a_2)_mx_2$ and $x_3x_{11}$ are $2$-torsion.
		    By \cite{bott1958}, there is no torsion in $H^*(\Omega G_2,\mathbb{Z})$.
		    Hence, for dimensional reasons, there is $b_{10}\in H^*(\Omega G_2 ; \mathbb{Z})$ and
		    \begin{equation*}
		       d^5((a_2)_tx_2)=(a_2)_{t-2}x_3x_{11}, \; d^9(b_{10})=a_2x_3, \; d^{11}(2b_{10})=x_{11}
		    \end{equation*}
		    for odd $t\geq 3$.
		    It can be shown that $b_{10}$ generates a divided polynomial algebra 
		    in the same way as for generators of $H^*(\Omega SU(3);\mathbb{Z})$
		    and $H^*(\Omega Sp(2);\mathbb{Z})$,
		    which completes the proof.
	    \end{proof}

	\section{Differentials in the diagonal map spectral sequence}\label{sec:evalSS}
		
			We begin by studying the differentials in the evaluation fibration of a simply connected, simple Lie group $G$ of rank $2$ with maximal torus $T$. 
			The argument is similar to that of \cite{cohololgy_Lprojective},
			in which the cohomology of the free loop spaces of spheres and complex projective space are calculated using spectral sequence techniques. 
			However the technical details in the case of the complete flag manifolds are considerably more complex.
			
			For a space $X$, the fibration $Map(I,X) \to X\times X$ is given by $\alpha \mapsto (\alpha(0),\alpha(1))$. Note that $Map(I,X)\simeq X$.
			It can be shown that $eval$ is a fibration homotopy equivalent to the diagonal map
			with its fibre $\Omega X$.
			In this section we compute the differentials in the cohomology Leray-Serre spectral sequence of this fibration in the case $X=G$.
			The aim is to compute $H^{*}(\Lambda G/T;\mathbb{Z})$.
			The fibration $eval \colon \Lambda X \to X$ is given by the evaluation of a loop at the  base point.
			The evaluation fibration is also a fibration with fibre
			$\Omega X$.
			This fibration is studied in Section \ref{sec:diff} by considering a map of fibrations from the evaluation fibration of $G$ to the diagonal fibration
			and hence the induced map on spectral sequences. For the rest of this section we consider the fibration
			
			\begin{equation}\label{eq:evalfib}
				\Omega (G/T) \to G/T \xrightarrow{\Delta} G/T\times G/T. 
			\end{equation}
			
			By extending the fibration $T \to G \to G/T$, we obtain the homotopy fibration sequence
			
			\begin{equation}\label{eq:SU/Tfib}
			\Omega G \to \Omega(G/T) \to T \to G.
			\end{equation}
			
			It is well known see \cite{CohomologyOmega(G/U)}, that the inclusion of the maximal torus into $G$ is null-homotopic.
			Hence there is a homotopy section $T \to \Omega G$, implying that
			\begin{equation}\label{eq:OmegaG/T}
			    \Omega (G/T) \simeq \Omega G \times T.
			\end{equation}
			All cohomology algebras of spaces in fibration~(\ref{eq:evalfib}) are known.
			As $G/T$ hence $G/T \times G/T$ are simply connected,
			the cohomology Leray-Serre spectral sequence of fibration~(\ref{eq:evalfib}), which we denote by $\{\bar{E}_r,\bar{d}^r\}$, 
			converges to $H^*(G/T;\mathbb{Z})$ with $\bar{E}_2$-page $\bar{E}^{p,q}_2=H^p(G/T \times G/T;H^q(\Omega(G/T);\mathbb{Z}))$.
			
			\subsection{Case $SU(3)/T^2$}
			
			When $G=SU(3)$, 
			using decomposition \eqref{eq:OmegaG/T}, we obtain the algebra isomorphism
			\begin{center}
			$H^*(\Omega(SU(3)/T^2);\mathbb{Z}) \cong H^*(\Omega(SU(3);\mathbb{Z}) \otimes H^*(T^2;\mathbb{Z})
			\cong \Gamma_{\mathbb{Z}}[x_2,x_4] \otimes \Lambda_{\mathbb{Z}}(y_1,y_2)$,
			\end{center}
			where $|x_2|=2$ and $|x_4|=4$ and $|y_1|=|y_2|=1$.
			Using the cohomology description \eqref{eq:H*SU/T}, we set
			\begin{center}
			$H^*(SU(3)/T^2;\mathbb{Z}) = \frac{\mathbb{Z}[\lambda_1,\lambda_2]}{\langle\sigma^{\lambda}_2,\sigma^{\lambda}_3\rangle}$
			\end{center}
			and
			\begin{center}
			$H^*(SU(3)/T^2 \times SU(3)/T^2;\mathbb{Z}) =
			\frac{\mathbb{Z}[\alpha_1,\alpha_2]}
			{\langle\sigma^{\alpha}_2,\sigma^{\alpha}_3\rangle} \otimes
			\frac{\mathbb{Z}[\beta_1,\beta_2]}
			{\langle\sigma^{\beta}_2,\sigma^{\beta}_3\rangle}$
			\end{center}
			where $|\alpha_i|=|\beta_i|=|\lambda_i|=2$ for each $i=1,2$.
			
				The following lemma determines the $\bar{d}^2$ differential on $\bar{E}_2^{*,1}$.
				We use the alternative basis
				\begin{center}
					$v_i=\alpha_i-\beta_i$ and $u_i=\beta_i$
				\end{center}
				for $H^{*}(SU(3)/T^2\times SU(3)/T^2;\mathbb{Z})$, where $i=1,2$. 
			
				\begin{lemma}\label{lemma:E^2_{*,1}d^2}
					In the cohomology Leray-Serre spectral sequence of fibration~(\ref{eq:evalfib}),
					there is a choice of basis $y_1,y_2$ such that
					\begin{center}
							$\bar{d}^2(y_i)=v_i$
					\end{center}
					for each $i=1,2$.
				\end{lemma}
				
				\begin{proof}
					For dimensional reasons, $\bar{d}^2$ is the only possible non-zero differential with codomain at any $\bar{E}_*^{2,0}$ and no non-zero differential have domain in any $\bar{E}_*^{2,0}$.
					As fibration~(\ref{eq:evalfib}) is the diagonal map $\Delta\colon SU(3)/T2\to SU(3)/T^2\times SU(3)/T^2$
					and the spectral sequence converges to $H^{*}(SU(3)/T^2)$,
					the image of $\bar{d}^2 \colon \bar{E}_2^{0,1} \to \bar{E}_2^{2,0}$ must be the kernel of the cup product on $H^*(SU(3)/T^2\times SU(3)/T^2;\mathbb{Z})$, 
					which is generated by $v_1,v_2$.
				\end{proof}

				\begin{theorem}\label{theorem:finaldiff}
    				In the spectral sequence $\{\bar{E}_r,\bar{d}^r\}$,
    				up to class representative and sign in $\bar{E}_2^{2,1}$  and $\bar{E}_2^{4,1}$, the non-trivial differentials are given by
    				\begin{equation*}
    				    \bar{d}^2(x_2)=y_1v_1+y_2v_2+y_1v_2+2y_1u_1+2y_2u_2+y_1u_2+y_2u_1
    				\end{equation*}
    				and
    				\begin{equation*}	
    				\bar{d}^4(x_4)=y_1v_1v_2+y_2v_2v_1+2y_1u_1v_2+2y_2u_2v_1+y_1v_1u_2+y_2v_2u_1+2y_1u_1u_2+2y_2u_2u_1+y_1u_2^2+y_2u_1^2.
    				\end{equation*}
				\end{theorem}
				
				\begin{proof}
				All differentials on $\alpha_i$ and $\beta_i$ are trivial for dimensional reasons.
				So the only remaining differentials left to determine are those on
				$\langle x_2 \rangle$ and $\langle x_4 \rangle$.
				For dimensional reasons, the elements $x_2,x_4$ cannot be the image of any differential.
				By Lemma~\ref{lemma:E^2_{*,1}d^2}, the generators $u_1,u_2$ must survive to the $\bar{E}_{\infty}$-page,
				so generators $x_2$ and $x_4$ cannot.	
				The image of $\bar{d}^{2}\colon\bar{E}_{2}^{0,2}\to \bar{E}_{2}^{2,1}$ will be a class in $\bar{E}_{2}^{2,1}$
				in the kernel of $\bar{d}^2$ generated by a single element,
				not in the image of elements generated by $y_i,u_i$ and $v_i$ alone.
				
				We have $\bar{d}^2(u_i)=\bar{d}^2(v_i)=0$ and by Lemma~\ref{lemma:E^2_{*,1}d^2} we may assume that $\bar{d}^2(y_i)=v_i$ for each $i=1,2$.
				Non-zero generators in $\bar{E}_2^{2,1}$ can be expressed in the form
				\begin{center}
					$y_k u_i$ or $y_k v_i$
				\end{center}
				and non-zero generators in $\bar{E}_2^{4,1}$ can be expressed in the form
				\begin{center}
					$y_k u_i u_j$, $y_k v_i u_j$ or $y_k v_i v_j$
				\end{center}
				for some $1\leq i,j,k \leq 2$.
                Notice that for $1\leq k,i \leq 2$
                \begin{align*}
                 &\bar{d}^2(y_1v_1+y_2v_2+y_1v_2+2y_1u_1+2y_2u_2+y_1u_2+y_2u_1)
                 \\
                 &=v_1^2+v_2^2+v_1v_2+2v_1u_1+2v_2u_2+v_1u_2+v_2u_1
                 \\
                 &=(\alpha_1^2-2\alpha_1\beta_1+\beta_1^2)
                 +(\alpha_2^2-2\alpha_2\beta_2+\beta_2^2)
                 +(\alpha_1\alpha_2-\alpha_1\beta_2-\alpha_2\beta_1+\beta_1\beta_2)
                 \\
                 &\;\;\;\;+2(\alpha_1\beta_1-\beta_1^2)
                 +2(\alpha_2\beta_2-\beta_2^2)
                 +(\alpha_1\beta_2-\beta_1\beta_2)
                 +(\alpha_2\beta_1-\beta_2\beta_1)
                 \\
                 &=\alpha_1^2+\alpha_2^2+\alpha_1\alpha_2-\beta_1^2-\beta_2^2-\beta_1\beta_2=0.
                \end{align*}
                
                In particular since $y_1v_1$ is not a term in $\bar d^2(y_1y_2)=y_2v_1-y_1v_2$,
                \begin{equation*}
                    y_1v_1+y_2v_2+y_1v_2+2y_1u_1+2y_2u_2+y_1u_2+y_2u_1
                \end{equation*}
                is a generator. Thus since it is the only remaining cycle that can be hit by $x_2$ in the kernel of $\bar d^2$ at $E^{2,1}_2$,
                \begin{equation*}
                \bar{d}^2(x_2)=y_1v_1+y_2v_2+y_1v_2+2y_1u_1+2y_2u_2+y_1u_2+y_2u_1
                \end{equation*}
                up to class representative and sign.
                For dimensional reasons and due to all lower rows except $\bar{E}_4^{*,2}$ and $\bar{E}_4^{*,1}$ being annihilated by the $\bar{d}^2$ differential, the only possible non-zero differential beginning at $x_{4}$,
				is $\bar{d}^{4}\colon\bar{E}_{4}^{0,4}\to \bar{E}_{4}^{4,1}$.
				The image of the differential $\bar{d}^{4}$ on $x_4$ will therefore be a class in $\bar{E}_{4}^{4,1}$
				in the kernel of $\bar{d}^2$ generated by a single elements,
				not in the image of elements generated by $y_i,u_i$ and $v_i$ alone.
                We see that
                \begin{align*}
                &\bar{d}^2(y_1v_1v_2+y_2v_2v_1+2y_1u_1v_2+2y_2u_2v_1+y_1v_1u_2+y_2v_2u_1+2y_1u_1u_2+2y_2u_2u_1+y_1u_2^2+y_2u_1^2)
                \\&=
                v_1^2v_2+v_2^2v_1+2v_1v_2u_1+2v_1v_2u_2
                +v_1^2u_2+v_2^2u_1+2v_1u_1u_2+2v_2u_1u_2+v_1u_2^2+v_2u_1^2
                \\&=
                (\alpha_1^2\alpha_2-\alpha_1^2\beta_2-2\alpha_1\alpha_2\beta_1
                +2\alpha_1\beta_1\beta_2+\alpha_2\beta_1^2-\beta_1^2\beta_2)
                \\ & \;\;\;\;
                +(\alpha_1\alpha_2^2-\alpha_2^2\beta_1-2\alpha_1\alpha_2\beta_2
                +2\alpha_2\beta_1\beta_2+\alpha_1\beta_2^2-\beta_1\beta_2^2)
                \\ & \;\;\;\; +2(\alpha_1\alpha_2\beta_1-\alpha_1\beta_1\beta_2-\alpha_2\beta_1^2+\beta_1^2\beta_2)
                +2(\alpha_1\alpha_2\beta_2-\alpha_1\beta_2^2-\alpha_2\beta_1\beta_2+\beta_1\beta_2^2)
                \\ & \;\;\;\;
                +(\alpha_1^2\beta_2-2\alpha_1\beta_1\beta_2+\beta_1^2\beta_2)
                +(\alpha_2^2\beta_1-2\alpha_2\beta_1\beta_2+\beta_1\beta_2^2)
                \\ & \;\;\;\;
                +2(\alpha_1\beta_1\beta_2-\beta_1^2\beta_2)
                +2(\alpha_2\beta_1\beta_2-\beta_1\beta_2^2)
                +(\alpha_1\beta_2^2-\beta_1\beta_2^2)
                +(\alpha_2\beta_1^2-\beta_1^2\beta_2)
                \\&=
                \alpha_1^2\alpha_2+\alpha_2^2\alpha_1
                -\beta_1^2\beta_2-\beta_2^2\beta_1=0.
                \end{align*}

                In particular since $y_1v_1u_2$ is not a term in $\bar d^2(y_1y_2u_2)$,
                \begin{equation*}
                    y_1v_1v_2+y_2v_2v_1+2y_1u_1v_2+2y_2u_2v_1+y_1v_1u_2+y_2v_2u_1
                    +2y_1u_1u_2+2y_2u_2u_1+y_1u_2^2+y_2u_1^2
                \end{equation*}
                is a generator. Thus since it is the only remaining cycle that can be hit by $x_4$ in the kernel of $\bar d^2$ at $E^{4,1}_2$,
                \begin{equation*}
                \bar{d}^4(x_4)=
                y_1v_1v_2+y_2v_2v_1+2y_1u_1v_2+2y_2u_2v_1+y_1v_1u_2+y_2v_2u_1+2y_1u_1u_2+2y_2u_2u_1+y_1u_2^2+y_2u_1^2
                \end{equation*}
                up to class representative and sign.
				\end{proof}

			\subsection{Case $Sp(2)/T^2$}
			Consider now $G=Sp(2)$.
    		By (\ref{eq:LoopSp}), we have
    		\begin{equation*}
    			H^*(\Omega(Sp(2)/T^2);\mathbb{Z})
    			\cong \Gamma_{\mathbb{Z}}[x_2,x_6] \otimes \Lambda_{\mathbb{Z}}(y_1,y_2)
    		\end{equation*}
    		where $\Gamma_{\mathbb{Z}}[x_2,x_4]$ is the integral divided polynomial algebra on variables $x_2,x_6$
    		with $|x_2|=2$ and $|x_6|=6$
    		and $\Lambda(y_1,y_2)$ is an exterior algebra generated by $y_1,y_2$
    		with $|y_1|=|y_2|=1$.
    	    By (\ref{eq:H*Sp/T}), the cohomology of $Sp(2)/T^2$ is
    		\begin{equation*}
    			H^*(Sp(2)/T^2);\mathbb{Z}) = \frac{\mathbb{Z}[\lambda_1,\lambda_2]}{\langle{\sigma^{\lambda^2}_1},{\sigma^{\lambda^2}_2}\rangle}
    		\end{equation*}
    		and
    		\begin{equation*}
    			H^*(Sp(2)/T^2 \times Sp(2)/T^2;\mathbb{Z}) = 
    			\frac{\mathbb{Z}[\alpha_1,\alpha_2]}
    			{\langle{\sigma^{\alpha^2}_1},\sigma^{\alpha^2}_2\rangle} \otimes
    			\frac{\mathbb{Z}[\beta_1,\beta_2]}
    			{\langle\sigma^{\beta^2}_1,\sigma^{\beta^2}_2\rangle}
    		\end{equation*}
    		for the cohomology of the base space and fibre of fibration (\ref{eq:evalfib}),
    		where $|\lambda_1|=|\alpha_i|=|\beta_i|=2$ for $i=1,2$.
    		Denote by $\{ \bar{E}^r,\bar{d}^r \}$ the cohomology Leray-Serre spectral sequence associated to fibration~(\ref{eq:evalfib}).
    		We again use the alternative basis
    		\begin{equation*}
    			v_i=\alpha_i-\beta_i\;\; \text{and} \;\; u_i=\beta_i
    		\end{equation*}
    		for $i=1,2$.
    		For exactly the same reasons as in Lemma~\ref{lemma:E^2_{*,1}d^2}, we get an equivalent lemma in the present case.
    		
    		\begin{lemma}\label{lemma:SpE^2_{*,1}d^2}
    			With the notation above, in the cohomology Leray-Serre spectral sequence of fibration~(\ref{eq:evalfib}),
					there is a choice of basis $y_1,y_2$ such that
					\begin{center}
							$\bar{d}^2(y_i)=v_i$
					\end{center}
					for each $i=1,2$.
    		\end{lemma}
    		We now prove an equivalent of Theorem~\ref{theorem:finaldiff} for $G=Sp(2)$.
    		\begin{theorem}\label{theorem:SpPathDiff}
    			In the spectral sequence $\{ E_r,d^r \}$ up to class representative and sign on $\bar{E}^2_{2,1}$ and $\bar{E}^2_{6,1}$, the only non-trivial differentials are given by
    			\begin{equation*}
    				\bar{d}^2(x_2)=y_1v_1+y_2v_2+2y_1u_1+2y_2u_2
    			\end{equation*}
    			and
    			\begin{equation*}
    				\bar{d}^6(x_6)
    				=y_1v_1^3+4y_1v_1^2u_1+6y_1v_1u_1^2+4y_1u_1^3
    			\end{equation*}
    		\end{theorem}
    		
    		\begin{proof}
    			All differentials on $\alpha_i$ and $\beta_i$ are trivial for dimensional reason.
				So the only remaining differentials left to determine are those on
				$\langle x_2 \rangle$ and $\langle x_6 \rangle$.
				For dimensional reasons, the elements $x_2,x_6$ cannot be in the image of any differential.
				By Lemma~\ref{lemma:SpE^2_{*,1}d^2}, the generators $u_1,u_2$ must survive to the $\bar{E}_{\infty}$-page,
				so generators $x_2$ and $x_6$ cannot.	
				The image of $\bar{d}^{2}\colon\bar{E}_{2}^{0,2}\to \bar{E}_{2}^{2,1}$ will be a class in $\bar{E}_{2}^{2,1}$
				in the kernel of $\bar{d}^2$ generated by a single element,
				not in the image of elements generated by $y_i,u_i$ and $v_i$ alone.
    			
				We have $\bar{d}^2(u_i)=\bar{d}^2(v_i)=0$ and by Lemma~\ref{lemma:SpE^2_{*,1}d^2} we  assume that $\bar{d}^2(y_i)=v_i$ for each $i=1,2$.
				Non-zero generators in $\bar{E}_2^{2,1}$ can be expressed in the form
				\begin{center}
					$y_k u_i$ or $y_k v_i$
				\end{center}
				and non-zero generators in $\bar{E}_2^{6,1}$ can be expressed in the form
				\begin{center}
					$y_k u_{i_1} u_{i_2} u_{i_3}$, $y_k v_{i_1} u_{i_2} u_{i_3}$, $y_k v_{i_1} v_{i_2} u_{i_3}$ or $y_k v_{i_1} v_{i_2} v_{i_3}$
				\end{center}
				for some $1\leq i_2,i_2,i_3,k \leq 2$.
                Notice that
                \begin{align*}
                    \bar{d}^2(y_kv_i)&=v_k^2
                    =\alpha_k^2-2\alpha_k\beta_k+\beta_k^2 \\
                    \bar{d}^2(y_ku_k)&=v_ku_k
                    =\alpha_k\beta_k-\beta_k^2
                \end{align*}
                so
                \begin{equation*}
                    \bar{d}^2(y_1v_1+y_2v_2+2y_1u_1+2y_2u_2)
                    =\alpha_1^2+\alpha_2^2-\beta_1^2-\beta_2^2=0.
                \end{equation*}
                 In particular since $y_1v_1$ is not a term in $\bar d^2(y_1y_2)=y_2v_1-y_1v_2$,
                \begin{equation*}
                    y_1v_1+y_2v_2+2y_1u_1+2y_2u_2
                \end{equation*}
               is a generator. Thus since it is the only remaining cycle that can be hit by $x_2$ in the kernel of $\bar d^2$ at $E^{2,1}_2$,
                \begin{equation*}
                    \bar{d}^2(x_2)=y_1v_1+y_2v_2+2y_1u_1+2y_2u_2
                \end{equation*}
                up to class representative and sign.
                Similarly,
                \begin{align*}
                    \bar{d}^2(y_1v_1^3)&=v_1^4
                    =\alpha_1^4
                    -4\alpha_1^3\beta_1
                    +6\alpha_1^2\beta_1^2
                    -4\alpha_1\beta_1^3
                    +\beta_1^4 \\
                    \bar{d}^2(y_1v_1^2u_1)&=v_1^3u_1
                    =\alpha_1^3\beta_1-3\alpha_1^2\beta_1^2
                    +3\alpha_1\beta_1^3-\beta_1^4 \\
                    \bar{d}^2(y_1v_1u_1^2)&=v_1^2u_1^2
                    =\alpha_1^2\beta_1^2-2\alpha_1\beta_1^3
                    +\beta_1^4 \\
                    \bar{d}^2(y_1u_1^3)&=v_1u_1^3
                    =\alpha_1\beta_1^3-\beta_1^4.
                \end{align*}
                Hence
                \begin{equation*}
                    \bar{d}^2(y_1v_1^3+4y_1v_1^2u_1+6y_1v_1u_1^2
                    +4y_1u_1^3)=
                    \alpha_1^4-\beta_1^4=0.
                \end{equation*}
                In particular since $y_1v_1^3$, is not a term in $\bar d^2(y_1y_2v_1^2)$,                 \begin{equation*}
                    y_1v_1^3+4y_1v_1^2u_1+6y_1v_1u_1^2
                    +4y_1u_1^3
                \end{equation*}
                is a generator. Thus since it is the only remaining cycle that can be hit by $x_6$ in the kernel of $\bar d^2$ at $E^{6,1}_2$,
                
                \begin{equation*}
                \bar{d}^6(x_6)=
                y_1v_1^3+4y_1v_1^2u_1+6y_1v_1u_1^2+4y_1u_1^3
                \end{equation*}
                up to class representative and sign.
    		\end{proof}

			\subsection{Case $G_2/T^2$}
			
			    Consider now $G=G_2$.
    			Using (\ref{eq:H*G2/T}) and Proposition~\ref{prop:LoopG2},
    			in the following argument we use the notation
    			for the cohomology of the base space and fibre in fibration (\ref{eq:evalfib})
				\begin{align*}
					H^*(G_2/T^2\times G_2/T^2;\mathbb{Z})
					&= \frac{\mathbb{Z}[\alpha_1,\alpha_2,l_3]}{\langle\sigma_2^\alpha,2l_3-\sigma_3^\alpha,l_3^2\rangle}
					\otimes
					\frac{\mathbb{Z}[\beta_1,\beta_2,s_3]}{\langle\sigma_2^\beta,2s_3-\sigma_3^\beta,s_3^2\rangle}\\
					H^*(\Omega(G_2/T^2);\mathbb{Z})
					&=\frac{\mathbb{Z}[({a}_2)_1,({a}_2)_2,\dots]}{\langle{a}_2^{m}-(m!/2^{\lfloor\frac{m}{2}\rfloor})({a}_2)_{m}\rangle}\otimes \Gamma_{\mathbb{Z}}[{b}_{10}]
					\otimes\Lambda_\mathbb{Z}(y_1y_2),\\
				\end{align*}
    			where $|\alpha_1|=|\beta_1|=|\alpha_2|=|\beta_2|=2$, $|y_1|=|y_2|=1$, $|a_2|=2$, $|b_{10}|=10$,
    			$|l_3|=|s_3|=6$.
			    Again we use the change of basis $u_i=\beta_i$ and $v_i=\alpha_i-\beta_i$.
				In addition we also make the change of basis
				\begin{equation*}
				    \theta=l_3-s_3 \text{ and } \psi=l_3.
				\end{equation*}
			
			\begin{theorem}\label{theorem:G2PathDiff}
			   In the spectral sequence $\{ \bar{E}_r,\bar{d}^r \}$ up to class representatives and sign on $\bar{E}^2_{2,1}$ and $\bar{E}^2_{10,1}$, the non-trivial differentials are given by
			   \begin{equation*}
			       \bar{d}^2(a_2)=
			       y_1(u_2+v_2+2u_1)+y_2(u_1+v_1+2u_2)
			   \end{equation*}
			   \begin{equation*}
			       \bar{d}^4(a_2(y_1(u_2+v_2+2u_1)+y_2(u_1+v_1+2u_2)))
			       =\theta
			   \end{equation*}
			   and
			   \begin{equation*}
			       \bar{d}^{10}(b_{10})=
			       y_1(\theta v_1^2+3\theta v_1u_1+3\theta u_1^2+2\psi v_1^2+3\psi v_1u_1+3y_1\psi u_1^2)
			   \end{equation*}
			\end{theorem}
			
			\begin{proof}
					Similarly to previous cases, the differential is given by
					\begin{equation*}
						\bar{d}^2(y_i)=v_i.
					\end{equation*}
					The differential on $a_2$ is obtained as in Theorem~\ref{theorem:finaldiff}.
				    and is given by
			        \begin{equation*}
			            \bar{d}^2({a}_2)
			            =y_1(u_2+v_2+2u_1)+y_2(u_1+v_1+2u_2)
				    \end{equation*}
				    which we denote by $\zeta$.
					
					The reminder of the augment is also similar to previous cases, however there are two exceptions.
					Firstly, since 
					\begin{equation*}
					    \frac{\mathbb{Z}[({a}_2)_1,({a}_2)_2,\dots]}{\langle{a}_2^{m}-(m!/2^{\lfloor\frac{m}{2}\rfloor})({a}_2)_{m}\rangle}
					\end{equation*}
					is not a divided polynomials algebra,
					\begin{align*}\label{eq:G2mult2diff}
						\bar{d}^2(({a}_2)_m)&=m\frac{2^{\lfloor \frac{m}{2}\rfloor}}{m!}\bar{d}^2({a}_2){a}_2^{m-1}
						=\frac{2^{\lfloor \frac{m}{2}\rfloor}}{(m-1)!}\bar{d}^2({a}_2)\frac{(m-1)!}{2^{\lfloor \frac{m-1}{2}\rfloor}}({a}_2)_{m-1} \\
						&=\frac{2^{\lfloor \frac{m}{2}\rfloor}}{2^{\lfloor \frac{m-1}{2}\rfloor}}({a}_2)_{m-1}\bar{d}^2({a}_2)=
						\begin{cases} 
						    2({a}_2)_{m-1}\zeta  &\mbox{for } m \text{ even}
		                    \\
		                    ({a}_2)_{m-1}d\zeta  &\mbox{for } m \text{ odd}
		                    \end{cases}
					\end{align*}					
					for each $m\geq 1$.
					Hence the differential is not surejctive and for odd $m\geq 1$, $[(a_2)_m\zeta]$ multiplicatively generates $2$-torsion on the $E_3$-page. 
					Secondly, the ideal $[2s_3-\alpha_1^3, 2l_3-\beta_1^3]$ does not correspond to any elements of the kernel of $\bar{d}^2$, however due to these relations
					\begin{equation*} 
						\bar{d}^2(y_1(v_1^2+3v_1u_1+3u_1^2))
						=2\theta
					\end{equation*}
					also multiplicatively generates $2$-torsion on the $\bar{E}_3$-page.
					If $a_2\zeta$ survived to the $\bar{E}_\infty$-page, this would imply that for
					dimensional reasons after resolving extension problems there would be 
					torsion class in $H^*(G_2/T^2;\mathbb{Z})$.
					However there is no torsion in $H^*(G_2/T^2;\mathbb{Z})$, so $[a_2\zeta]$ must be trivial by the $\bar{E}_\infty$-page.
					Hence for dimensional reason the only possibility is
					\begin{equation*}
						\bar{d}^4(a_2\zeta)=\theta
					\end{equation*}
					up to class representative and sign.
					Since $[a_2\zeta]$ and $[\theta]$ generate all $2$-torsion, for $r\leq 9$, $\bar{d}^r(b_{10})=0$ and there is no torsion by the $\bar{E}_5$-page.		
					In particular, we determine the differentials on $b_{10}$
					in the same way as in previous cases.
					Notice that
					\begin{align*}
						\bar{d}^2(y_1\theta v_1^2)& =v_1^3\theta	 =2s_3^2-4s_3l_3-3s_3\alpha_2^2\beta_1+3\alpha_1^2l_3\beta_1+3s_3\alpha_1\beta_1^2-3\alpha_1l_3\beta_1^2+2l_3^2 ,\\
						\bar{d}^2(y_1\theta v_1u_1)&=\theta v_1^2u_1=s_3\alpha_1^2\beta_1-2s_3\alpha_1\beta_1^2-\alpha_1^2l_3\beta_1+2\alpha_1l_3\beta_1^2+2s_3l_3-2l_3^2 ,\\
						\bar{d}^2(y_1\theta u_1^2)& =\theta v_1u_1^2=s_3\alpha_1\beta_1^2-2s_3l_3-\alpha_1l_3\beta_1^2+2l_3^2 ,\\
						\bar{d}^2(y_1\psi v_1^2)&    = v_1^3\psi	      =2s_3l_3-3\alpha_1^2l_3\beta_1+3\alpha_1l_3\beta_1^2-2l_3^2 ,\\
						\bar{d}^2(y_1\psi v_1u_1)&   =v_1^2 \psi u_1  =\alpha_1^2l_3\beta_1-2\alpha_1l_3\beta_1^2+2l_3^2 ,\\
						\bar{d}^2(y_1\psi u_1^2)&    =v_1\psi u_1^2   =\alpha_1l_3\beta_1^2-2l_3^2.
					\end{align*}
					Hence
					\begin{equation*}
						\bar{d}^2(y_1(\theta v_1^2+3\theta v_1u_1+3\theta u_1^2+2\psi v_1^2+3\psi v_1u_1+3\psi u_1^2))=2s_3^2-2l_3^2=0
					\end{equation*}
					and for the same reasons as in previous cases 
					\begin{equation*}
    					\bar{d}^{10}(b_{10})=
    			         y_1(\theta v_1^2+3\theta v_1u_1+3\theta u_1^2+2\psi v_1^2+3\psi v_1u_1+3y_1\psi u_1^2)
					\end{equation*}
					up to class representative and sign.
				\end{proof}

			\section{Differentials in the cohomology Leray-Serre spectral sequence of the evaluation fibration}\label{sec:diff}
			
			Throughout the following argument we consider the map $\phi$ of fibrations between the evaluation fibration of complete flag manifold $G/T$  
			and the diagonal fibration
			given by the following commutative diagram
			\begin{equation}\label{fig:fibcd}
						\xymatrix{
							{\Omega(G/T)} \ar[r]^(.5){} \ar[d]^(.45){id} & {\Lambda(
							G/T)} \ar[r]^{eval} \ar[d]^(.45){eval}  & {G/T} \ar[d]^(.45){\Delta} \\
							{\Omega(G/T)} \ar[r]^(.45){}   							 & {G/T} 	 \ar[r]^-(.5){\Delta}  					 & {G/T\times G/T}.}
			\end{equation}
			Since we assume that $G$, hence $G/T$ is simply connected, the cohomology Leray-Serre spectral sequence $\{ E_r,d^r \}$ associated with the evaluation fibration converges.
			Hence $\phi$ induces a map of spectral sequences $\phi^*:\{ \bar{E}_r,\bar{d}^r \} \to \{ E_r,d^r \}$. 
			More precisely, for each $r\geq 2$ and $a,b \in \mathbb{Z}$ there is a commutative diagram
			\begin{equation}\label{fig:phicd}
						\xymatrix{
							{\bar{E}_r^{a,b}} \ar[r]^(.4){\bar{d}^r} \ar[d]^(.46){\phi^*} & {\bar{E}_r^{a+r,b-r+1}} \ar[d]^(.46){\phi^*}  \\
							{E_r^{a,b}} \ar[r]^(.4){d^r}   			& {E_r^{a+r,b-r+1}}}
					\end{equation}
			where $\phi^*$, for each $r$, is the induced map on the homology of the previous page, beginning as the map induced on the tensor on the $E_2$-pages
			by the maps $id \colon \Omega(G/T)\to\Omega(G/T)$ and $\Delta \colon G/T \to G/T \times G/T$.

			\subsection{Case $SU(3)/T^2$}\label{sec:fdiffSU}
			Let $G=SU(3)$.
			By (\ref{eq:OmegaG/T}) and \eqref{eq:H*SU/T}, we have
			\begin{equation*}
			H^*(\Omega(SU(3)/T^2);\mathbb{Z})\cong \Gamma_{\mathbb{Z}}(x'_2,x'_4)\otimes\Lambda_{\mathbb{Z}}(y'_1,y'_{2})
			\end{equation*}
			and
			\begin{equation*}
			\;\;\; H^*(SU(3)/T^2;\mathbb{Z})\cong \frac{\mathbb{Z}[\gamma_1,\gamma_2]}{\langle\sigma_{2}^\gamma,\sigma_{3}^\gamma\rangle}
			\end{equation*}
			where $|y'_i|=1, |\gamma_j|=2, |x'_{2i}|=2i$ for each $1\leq i\leq n,1\leq j\leq n+1$. 
			Next we determine all the differentials in $\{E_r,d^r \}$ when $G=SU(3)$.	
			
			\begin{theorem}\label{theorem:allDiff}
			The only non-zero differentials on generators of the $E_2$-page of $\{E_r,d^r \}$ are, up to class representative and sign, given by
				\begin{equation*}
					d^{2}(x'_{2})
					=
					y'_1(2\gamma_1+\gamma_2)
					+y'_2(\gamma_1+2\gamma_2)
				\end{equation*}
				and
				\begin{equation*}
					d^{4}(x'_4)
					=y'_1(\gamma_2^2+2\gamma_1\gamma_2)
					+y'_2(\gamma_1^2+2\gamma_1\gamma_2).
				\end{equation*}
			\end{theorem}
			
			\begin{proof}
				The identity $id \colon \Omega(SU(3)/T^2)\to\Omega(SU(3)/T^2)$ induces the identity map on cohomology, while the diagonal map $\Delta \colon SU(3)/T^2 \to SU(3)/T^2 \times SU(3)/T^2$ induces the cup.
				Hence, by choosing generators in $\{E_2,d^2 \}$, we may assume that for $i=1,2$
				\begin{equation*}
    				\phi^*(y_i)
    				=y'_i,\;\;\phi^*(x_i)
    				=x'_i\;\;\text{and}\;\;\phi^*(\alpha_i)
    				=\phi^*(\beta_i)
    				=\phi^*(u_i)
    				=\gamma_i
				\end{equation*}
			    Therefore, $\phi^*(v_i)=0$ for $i=1,2$.
				For dimensional reasons, the only possibly non-zero differential on generators $y'_i$ in $\{ E_r,d^r \}$ is $d^{2}$.
				However using commutative diagram (\ref{fig:phicd}) and Lemma~\ref{lemma:E^2_{*,1}d^2}, we have
				\begin{center}
				$d^2(y'_i)=d^2(\phi^*(y_i))=\phi^*(\bar{d}^2(y_i))=\phi^*(v_i)=0$.
				\end{center}
				Using commutative diagram (\ref{fig:phicd}) and Theorem~\ref{theorem:finaldiff}, we have up to class representative and sign
				\begin{align*}
				    d^2(x'_2)
				    &=\phi^*(\bar{d}^2(x_2))\\
				    &=\phi^*(y_1v_1+y_2v_2+y_1v_2+2y_1u_1
				    +2y_2u_2+y_1u_2+y_2u_1)\\
					&=2y'_1\gamma_1+2y'_2\gamma_2
					+y'_1\gamma_2+y'_2\gamma_1
				\end{align*}
				and
				\begin{align*}
    				d^4(x'_4)&=\phi^*(\bar{d}^4(x_4)) \\
    				&=\phi^*(y_1v_1v_2+y_2v_2v_1+2y_1u_1v_2
    				+2y_2u_2v_1+y_1v_1u_2 \\
    				&\;\;\;\;+y_2v_2u_1+2y_1u_1u_2
    				+2y_2u_2u_1+y_1u_2^2+y_2u_1^2) \\
    				&=2y'_2\gamma_1\gamma_2+2y'_1\gamma_1\gamma_2
					+y'_2\gamma_1^2+y'_1\gamma_2^2.
				\end{align*}
				Differentials on generators $\gamma_i$ for each $i=1,2$ are zero for dimensional reasons.
			\end{proof}

            \subsection{Case $Sp(2)/T^2$}
			
    		Just as we did in Theorem~\ref{theorem:allDiff},
    		we can now use the results of Theorem~\ref{theorem:SpPathDiff} and diagram (\ref{fig:fibcd}) to deduce the differentials in the 
    		cohomology Leray-Serre spectral sequence $\{ E_r,d^r \}$ associated to the evaluation fibration of $Sp(2)/T^2$.
    		For the rest of the section, we denote the cohomology algebras of the base space and fibre of the
    		evaluation fibration by
    		\begin{equation*}\label{eq:spLoppCoHomology}
    			H^*(\Omega(Sp(2)/T^2);\mathbb{Z})=\Gamma_{\mathbb{Z}}(x'_2,x'_6)\otimes \Lambda_{\mathbb{Z}}(y'_1,y'_2)
    		\end{equation*}
    		and
    		\begin{equation*}
    			H^*(Sp(2)/T^2;\mathbb{Z})=\frac{\mathbb{Z}[\gamma_1,\gamma_2]}{\langle\sigma_1^2,\dots,\sigma_{n}^2\rangle}
    		\end{equation*}
    		where $|{y'}_1|=1=|y'_2|$, $|\gamma_1|=2=|\gamma_2|$, $|x'_2|=2$,  $|x'_6|$=6 and $\sigma_1^2,\dots,\sigma_n^2$ 
    		are the elementary symmetric polynomials in variables $\gamma_1^2,\gamma_2^2$.
    		
    		\begin{theorem}\label{theorem:SpDiff}
    			The only non-zero differentials on generators of the $E_2$-page of $\{E_r,d^r \}$
    			are, up to class representative and sign, given by
    			\begin{equation*}
    				d^2(x'_2)=2y'_1\gamma_1+2y'_2\gamma_2
    			\end{equation*}
    			and
    			\begin{equation*}
    			    d^6(x'_6)=y'_1\gamma_1^3.
    			\end{equation*}
    		\end{theorem}
    		
    		\begin{proof}
    			For the same reasons as in the proof of Theorem~\ref{theorem:allDiff}, we have for $i=1,2$
    			\begin{equation*}
    				\phi^*(y_i)=y'_i,\;\;\phi^*(x_i)
    				=x'_i\;\;\text{and}\;\;\phi^*(\alpha_i)
    				=\gamma_i
    				=\phi^*(\beta_i)
    				=\phi^*(u_i),\;\;\text{so}\;\;\phi^*(v_i)=0. 
    			\end{equation*}
    			Hence by the same arguments as in the proof of Theorem~\ref{theorem:allDiff}, we have 
    			\begin{equation*}
    				d^r(y'_i)=0 \text{ and } d^r(\gamma_i)=0
    			\end{equation*}
    			and the image of $d^r$ on generators $x'_2,x'_6$
    			is determined by those summands in the image of $\bar{d}^2$ on $x_2,x_6$ given in Theorem~\ref{theorem:SpPathDiff} containing no $v_i$, replacing $u_i$ with $\gamma_i$ and $y_i$ with $y'_i$.
    			This proves the statement.
    		\end{proof}

			\subsection{Case $G_2/T^2$}\label{sec:DiffG2}
			
    			To obtain the differentials in the Leray-Serre spectral sequence of the evaluation fibration of $G_2/T^2$, as in previous flag manifolds, we consider the map of Leray-Serre spectral sequences induced by (\ref{fig:phicd}).
    				In the following argument using (\ref{eq:H*G2/T}) and Theorem~\ref{prop:LoopG2} we have
    				\begin{equation*}
    					H^*(G_2/T^2;\mathbb{Z})= \frac{\mathbb{Z}[\gamma_1,\gamma_2,t_3]}{\langle\sigma_1^\gamma,\sigma_2^\gamma,2t_3-\sigma_3^\gamma,t_3^2\rangle}
    				\end{equation*}
    				and
    				\begin{equation*}
    					H^*(\Omega(G_2/T^2);\mathbb{Z})	=\frac{\mathbb{Z}[(a'_2)_1,(a'_2)_2,\dots]}{\langle{a'}_2^{m}-(m!/2^{\lfloor\frac{m}{2}\rfloor})(a'_2)_{m}\rangle}
    					\otimes \Gamma_{\mathbb{Z}}[b'_{10}] 
    					\otimes \Lambda_{\mathbb{Z}}(y'_1,y'_2)
    				\end{equation*}
    				where $=|\gamma_1|=2=|\gamma_2|$, $|y_1|=1=|y_2|$, $|a'_2|=2$, $|b'_{10}|=10$, $\sigma_i$ are elementary symmetric polynomials of degree $i$ in bases $\gamma_1\gamma_2,\gamma_3$ with $-\gamma_3=\gamma_1+\gamma_2$.
			
			        \begin{theorem}\label{theorem:DiffG2}
					The cohomology Leray-Serre spectral sequence $\{E_r,d^r \}$ associated to the evaluation fibration of $G_2/T^2$
					has, up to class representative, the only non-trivial differentials 
					\begin{equation*}
						d^2(a'_2)
						=y'_1(2\gamma_1+\gamma_2)
						+y'_2(\gamma_1+2\gamma_2)
					\end{equation*}
					and
					\begin{equation*}
						d^{10}(b'_{10})=3y'_1t_3\gamma_1^2.
					\end{equation*}
				\end{theorem}
			
			    \begin{proof}
			        We deduce the differentials in $\{E_r,d^r \}$ using the notation and results of Theorem~\ref{theorem:G2PathDiff}.
				    For the same reasons as in the proof of Theorem~\ref{theorem:allDiff}, we have
    			    \begin{equation*}
        				\phi^*(y_i)=y'_i,\;\;\phi^*(x_i)
        				=x'_i\;\;\text{and}\;\;\phi^*(\alpha_i)
        				=\gamma_i
        				=\phi^*(\beta_i)
        				=\phi^*(u_i),\;\;\text{so}\;\;\phi^*(v_i)=0.
    			    \end{equation*}
    			    Recall that $\theta=\bar{d}^4(a_2(y_1(u_2+v_2+2u_1)+y_2(u_1+v_1+2u_2)))$ and $\psi=l_3$.
    		Then
    			    \begin{equation*}
    			        \phi^*(\theta)=0 \text{ and } \phi^*(\psi)=t_3.
    			    \end{equation*}
        			Hence by the same arguments used in the proof of Theorem~\ref{theorem:allDiff}, we have 
        			\begin{equation*}
        				d^r(\theta)=0,
        				d^r(y'_i)=0 
        				\text{ and }
        				d^r(\gamma_i)=0.
        			\end{equation*}
				    Using the results of Theorem~\ref{theorem:G2PathDiff},
				    we  deduce the differentials in $\{E_r,d^r \}$.
				    Recall that $\zeta=y_1(u_2+v_2+2u_1)+y_2(u_1+v_1+2u_2)$. Since $\bar d^4$ is non-trivial only on $a_2\zeta$, and $\phi^*\bar d^4(a_2\zeta)=\phi^*(\theta)=0$, the differential $d^4$ is trivial.

				    The image of $d^r$ on generators $a'_2,b'_{10}$
        			is determined by those summands in the image of $\bar{d}^2$ on $a_2,b_{10}$ given in Theorem~\ref{theorem:G2PathDiff} containing no $v_i$ or $\theta$, and replacing $u_i$ with $\gamma_i$, $y_i$ with $y'_i$ and $\phi$ with $t_3$.
        			This gives the result stated in the theorem.
				\end{proof}

	\section{Free loop cohomology of complete flag manifolds of simple Lie groups of rank~$2$}
	
		In this section we calculate the cohomology of the free loop space of all complete flag manifolds arising form simple Lie groups of rank $2$.

		\subsection{Free loop cohomology of $SU(3)/T^2$}
		
			\begin{theorem}\label{theorem:L(SU(3)/T^2)}
				The integral algebra structure of the $E_\infty$-page of the Leray-Serre spectral sequence associated to the evaluation fibration of $SU(3)/T^2$
				is $A/I$, where
				\begin{align*}
					A=\Lambda_\mathbb{Z}( & \gamma_i,\; (x_4)_m,\; y_i,\; (x_2)_my_1y_2,\; (x_2)_m(y_1(\gamma_1+\gamma_2)-y_2\gamma_2), \; (x_2)_m(y_2(\gamma_1+\gamma_2)-y_1\gamma_1), \\ &
					(x_2)_m(2y_1\gamma_1^2+y_2\gamma_1^2),\;
                    (x_2)_m\gamma_1^2\gamma_2, \;
                    (x_2)_m\gamma_1^3, \;
                    (x_2)_m(\gamma_1^2+\gamma_2^2+\gamma_1\gamma_2)
                    )
				\end{align*}
				and
				\begin{align*}
					I= \langle &((x_2)_1^m-m!(x_2)_m)j, \; (x_4)_1^m-m!(x_4)_m, \;
					\\ &(x_2)_a(\gamma_1^2+\gamma_2^2+\gamma_1\gamma_2),\;(x_2)_a\gamma_1^3,\; (x_2)_a(y_2(\gamma_1+2\gamma_2)+y_1(2\gamma_1+\gamma_2)) \rangle
				\end{align*}
				where $m\geq 1$, $a\geq 0$,
				$|\gamma_i|=2$, $|y_i|=1$, $|(x_2)_k|=2k$, $|(x_4)_k|=4k$ and
				\begin{align*}
				    j\in\{ y_1y_2,\; y_1(\gamma_1+\gamma_2)-y_2\gamma_2,\; y_2(\gamma_1+\gamma_2)-y_1\gamma_1,\; 2y_1\gamma_1^2+y_2\gamma_1^2,\; \gamma_1^2+\gamma_2^2+\gamma_1\gamma_2,\; \gamma_1^2\gamma_2,\; \gamma_1^3 \}
				\end{align*}
				for $1\leq i \leq 2$ and $1\leq k$.
			\end{theorem}
			
			\begin{proof}
				We consider the cohomology Leray-Serre spectral sequence $\{ E_r,d^r \}$ associated to the evaluation fibration of $SU(3)/T^2$ studied in Section \ref{sec:diff}, that is,
				\begin{equation*}
					\Omega(SU(3)/T^2)\to\Lambda(SU(3)/T^2)\to SU(3)/T^2.
				\end{equation*}
			     By (\ref{eq:H*SU/T}), the integral cohomology of the base space $SU(3)/T^2$ is given by
				\begin{equation*}
				    \frac{\mathbb{Z}[\gamma_1,\gamma_2]}{\langle\gamma_1^2+\gamma_1\gamma_2+\gamma_2^2,\gamma_1^3\rangle}
				\end{equation*}
				where $|\gamma_1|=|\gamma_2|=2$.
				The integral cohomology of the fibre $\Omega(SU(3)/T^2)$ is given by
				\begin{equation*}
				    \Lambda_\mathbb{Z}(y_1,y_2)\otimes\Gamma_\mathbb{Z}[x_2,x_4]
				\end{equation*}
				where $|y_1|=|y_2|=1$, $|x_2|=2$ and $|x_4|=4$.
				Additive generators on the $E_2$-page of the spectral sequence are given by representative elements of the form
				\begin{equation*}
					(x_2)_a(x_4)_bP, \;\; (x_2)_a(x_4)_by_iP, \;\; (x_2)_a(x_4)_by_1y_2P
				\end{equation*}
				where $0\leq a,b$, $1\leq i \leq 2$ and $P\in \mathbb{Z}[\gamma_1,\gamma_2]$ is a monomial of degree at most $3$.
				By Theorem~\ref{theorem:allDiff}, the only non-zero differentials are $d^2$ and $d^4$,
				which are non-zero only on generators $x_2$ and $x_4$, respectively.
				The differentials up to sign are given by
				\begin{equation*}\label{eq:DifferentialsSU}
		        	d^2(x_2)=y_1(2\gamma_1+\gamma_2)+y_2(\gamma_1+2\gamma_2), \;\;\;
					d^4(x_4)=[y_1(\gamma_1^2+2\gamma_1\gamma_2)+y_2(\gamma_2^2+2\gamma_1\gamma_2)].
				\end{equation*}
				As
				\begin{align*}
					&d^2(x_2(\gamma_1+\gamma_2)) \\
					&=y_1(2\gamma_1^2+3\gamma_1\gamma_2+\gamma_2^2)+y_2(\gamma_1^2+3\gamma_1\gamma_2+2\gamma_2^2) \\
					&=y_1(\gamma_1^2+2\gamma_1\gamma_2)+y_2(\gamma_2^2+2\gamma_1\gamma_2) \\
					&=d^4(x_4)
				\end{align*}
				where the second equality is given by subtracting elements of the symmetric ideal $y_i(\gamma_1^2+\gamma_2^2+\gamma_1\gamma_2)$ for $i=1,2$ 
				from $y_1(2\gamma_1^2+3\gamma_1\gamma_2+\gamma_2^2)+y_2(\gamma_1^2+3\gamma_1\gamma_2+2\gamma_2^2)$, the differential $d^4$ is trivial, and the spectral sequence converges at the third page.
				
				The monomial generators $\gamma_i$, $x_4$, $y_i$ and $(x_2)_my_1y_2$ occur in $E_2^{*,0}$ or $E_2^{0,*}$ and are always in the kernel of the differentials,
				so are algebra generators of the $E_\infty$-page.
				All relations on the $E^3$-page coming from the relations in
				        $H^*(\Omega(SU(3)/T^2);\mathbb{Z}): 
                            (x_2)_1^m-m!(x_2)_m$ and $ 
                            (x_4)_1^m-m!(x_4)_m$, 
                       the relations in  $H^*(SU(3)/T^2;\mathbb{Z}):
                           \gamma_1^2+\gamma_1\gamma_2+\gamma_2^2$ and 
                            $\gamma_1^3$,
                        or are in the image of $d^2: y_1(\gamma_1+2\gamma_2)+y_2(2\gamma_1+\gamma_2)$
				hold on the $E_\infty$-page and therefore are in $I$ if they are in $A$.
				For this reason, since $E^2$ generator $(x_2)_m$ will not be in $A$, all generators of $I$ must be considered up to multiple of $(x_2)_a$.
			    In addition, we add $((x_2)_1^m-m!(x_2)_m)j$
				for $j$ such that $(x_2)_mj$ is a generator,
				rather than $(x_2)_1^m-m!(x_2)_m$ to $I$. 
		        It remains to determine all generators of $A$.
				
				The elements on the $E_2$-page of the form $(x_2)_a(x_4)_by_1y_2P$, $(x_4)_by_iP$ and $(x_4)_bP$ are in the kernel of $d^2$ and generated by $y_1y_2(x_2)_m, (x_4)_m, \gamma_1$ and $\gamma_2$.
				
				Let $\phi\colon\mathbb{Z}[(x_4)_m,(x_2)_m,\gamma_1,\gamma_2] \to \mathbb{Z}[(x_4)_m,(x_2)_m,y_1,y_2,\gamma_1,\gamma_2]$ be the map defined by $d^2$ so that the following digram commutes
				\begin{equation*}
				    \xymatrix{
							{\ker(\phi)} \ar[r]^(.5){} \ar[d]^(.45){q} & {\mathbb{Z}[(x_4)_m,(x_2)_m,\gamma_1,\gamma_2]} \ar[r]^(0.45){\phi} \ar[d]^(.45){q}  & {\mathbb{Z}[(x_4)_m,(x_2)_m,y_1,y_2,\gamma_1,\gamma_2]} \ar[d]^(.45){q} \\
							{\ker(d^2)} \ar[r]^(.45){} & {E_2} \ar[r]^-(.5){d^2} & {E_2}}
				\end{equation*}
				where $q$ is the quotient map by the symmetric polynomials and the divided polynomial relations
				\begin{equation*}
				    \gamma_1^2+\gamma_1\gamma_2+\gamma_2^2,\;  \gamma_1^3,\; (x_2)_1^m-m!(x_2)_m \text{ and } (x_4)_1^m-m!(x_4)_m.
				\end{equation*}
				By direct calculation, it can be seen that the kernel of $\phi$ is generated by $\gamma_1$,  $\gamma_2$ and $(x_4)_m$ which are also in the kernel of $d^2$.
				Thus the remaining elements of the kennel of $d^2$ of the form $(x_2)_m(x_4)_bP$ are elements of the ideal
				\begin{equation}\label{eq:IdealIntyi}
				    q\phi^{-1}(\langle y_2(2\gamma_2+\gamma_1)+y_1(\gamma_2+2\gamma_1)\rangle \cap\langle\gamma_2^2+\gamma_2\gamma_1+\gamma_1^2,\; \gamma_1^3 \rangle)
				\end{equation}
				where $q\phi^{-1}(y_1(2\gamma_2+\gamma_1)+y_1(\gamma_2+2\gamma_1))$ spans $\text{Im}(\phi)$.
				Intersection (\ref{eq:IdealIntyi}) can be computed by Gr\"{o}bner basis to show it contains no generators with $\gamma_i$ term lower than degree $4$.
				Hence the generators of the of kernel are of the form $(x_2)_m(x_4)_bP$ where the degree of $P$ is $3$.
				Since it is not yet contained in the set of generators, we add relation
				$(x_2)_m(\gamma_1^2+\gamma_2^2+\gamma_1\gamma_2)$
				as a generator in order to minimize the number of algebra generators required.				
				
				Let $\psi\colon\mathbb{Z}[(x_4)_m,(x_2)_my_1, (x_2)_my_2,\gamma_1,\gamma_2] \to \mathbb{Z}[(x_4)_m,(x_2)_m,y_1,y_2,\gamma_1,\gamma_2]$ 
				be the map defined by $d^2$ so that the following digram commutes
				\begin{equation*}
				    \xymatrix{
							{\ker(\psi)} \ar[r]^(.5){} \ar[d]^(.45){q} & {\mathbb{Z}[(x_4)_m,(x_2)_my_1, (x_2)_my_2,\gamma_1,\gamma_2]} \ar[r]^(0.525){\psi} \ar[d]^(.45){q}  & {\mathbb{Z}[(x_4)_m,(x_2)_m,y_1,y_2,\gamma_1,\gamma_2]} \ar[d]^(.45){q} \\
							{\ker(d^2)} \ar[r]^(.45){} & {E_2} \ar[r]^-(.5){d^2} & {E_2}}
				\end{equation*}
				where $q$ is the quotient map by relations
				\begin{equation*}
				    \gamma_1^2+\gamma_1\gamma_2+\gamma_2^2,\; \gamma_1^3,\; (x_2)_1^m-m!(x_2)_m \text{ and } (x_4)_1^m-m!(x_4)_m.
				\end{equation*}
				By direct calculation, it can be seen that the kernel of $\psi$ is generated by
				$\gamma_1$, $\gamma_2$, $(x_4)_m$ and the image of $d^2$,
			 $(x_2)_m(y_2(2\gamma_2+\gamma_1)+y_1(\gamma_2+2\gamma_1))$ 
				
				Thus the remaining element of the kernel of $d^2$ of the form $(x_2)_m(x_4)_by_iP$ are elements of the ideal
				\begin{equation}\label{eq:IdealInt}
				    q\psi^{-1}(\langle y_1y_2(2\gamma_2+\gamma_1),\; y_1y_2(\gamma_2+2\gamma_1)\rangle \cap \langle \gamma_2^2+\gamma_2\gamma_1+\gamma_1^2,\; \gamma_1^3 \rangle)
				\end{equation}
				where $q\psi^{-1}(y_1y_2(2\gamma_2+\gamma_1),\;y_1y_2(\gamma_2+2\gamma_1))$ spans $\text{Im}(\psi)$.
				
				Intersection (\ref{eq:IdealInt}) can be computed by Gr\"{o}bner basis with respect to the lexicographic monomial ordering $y_2>y_1>\gamma_2>\gamma_1>(x_4)_m>(x_2)_m$ yielding the ideal
				\begin{equation}\label{eq:Idealyi}
				  \langle y_1y_2(\gamma_2^2+\gamma_2\gamma_1+\gamma_1^2),\; 3y_1y_2\gamma_1^3,\; y_1y_2(\gamma_2\gamma_1^3+2\gamma_1^4) \rangle.
				\end{equation}
				Generators of the $q\psi^{-1}$ image of (\ref{eq:Idealyi}) will up to multiple of $(x_2)_m$, generate the kernel of $d^2$ with elements of the form $(x_2)_m(x_4)_by_iP$.
				The $q\psi^{-1}$ image of generator $y_1y_2(\gamma_2^2+\gamma_2\gamma_1+\gamma_1^2)$ of (\ref{eq:Idealyi}) is $[y_2(\gamma_1+\gamma_2)+y_1\gamma_1]$. 
				To write the image of $d^2$ in terms of generators we add
				\begin{equation*}
				    (x_2)_m(y_2(\gamma_1+\gamma_2)+y_1\gamma_1) \text{ and } (x_2)_m(y_1(\gamma_1+\gamma_2)+y_2\gamma_2)
				\end{equation*}
				as generators of the algebra.
				The $q\psi^{-1}$ image of generator $3y_1y_2\gamma_1^3$ of (\ref{eq:Idealyi}) is $[2y_1\gamma_1^2+y_2\gamma_1^2]$.
				Hence we take 
				\begin{equation*}
				    (x_2)_m(2y_1\gamma_1^2+y_2\gamma_1^2)
				\end{equation*}
				as generators of the algebra.
				Finally the $q\psi^{-1}$ image of generator $y_1y_2(\gamma_2\gamma_1^3+2\gamma_1^4)$ of (\ref{eq:Idealyi}) is trivial.
			\end{proof}
			
			\begin{theorem}\label{corollary:L(SU(3)/T^2)}
			    The cohomology $H^*(\Lambda(SU(3)/T^2);\mathbb{Z})$ is isomorphic as a module to $A/I$
			    given in Theorem~\ref{theorem:L(SU(3)/T^2)}. In addition there is no multiplicative extension problem on
			    the sub-algebra generated by $\gamma_a,\gamma_2$, the sub-algebra generated by $y_1,y_2,(x_4)_m$
			    and no multiplicative extension on elements $y_i\gamma_j$ for $1\leq i,j\leq 2$
			\end{theorem}
			
			\begin{proof}
			    Beginning with the structure of the $E^\infty$-page, given in Theorem~\ref{theorem:L(SU(3)/T^2)},
			    of the Leray-Serre spectral sequence of the evaluation fibration, 
			    we first consider the additive extension problems. 
			    
			    Since there is no torsion produced by the divided polynomial relations
			    \begin{equation*}
			        (x_2)_1^m-m!(x_2)_m)j \text{ and } (x_4)_1^m-m!(x_4)_m
			    \end{equation*}
				a Gr\"{o}bner basis of the ideal
				\begin{equation*}
				    \langle y_1^2,\;y_2^2,\;
				    \gamma_1^2+\gamma_1\gamma_2+\gamma_2^2,\gamma_1^3,\;
				    y_2(\gamma_1+2\gamma_2)+y_1(2\gamma_1+\gamma_2),\;
				    y_1y_2(2\gamma_1+\gamma_1),\;y_1y_2(2\gamma_2+\gamma_1)\rangle
				\end{equation*}
				has elements with terms containing  coefficient $1$ or all the coefficients are $3$.
				Hence by Theorem~\ref{thm:GrobnerOver}, all the torsion on the $E_\infty$-page of the spectral sequence is $3$-torsion.
				In order to resolve any additive extension problems we consider the spectral sequence $\{ E_r, d^r \}$ over the field of characteristic $3$.
				
				None of the generators in the integral spectral sequence are divisible by $3$,
				hence in the modulo $3$ spectral sequence all of the integral generators remain non-trivial. 
				In addition when the kernel of $d^2$ at $E_2^{p,q}$ is all of $E_2^{p,q}$, the free rank plus torsion rank in the integral spectral sequence
				must be greater than or equal to the rank in the modulo $3$ spectral sequence.
				So in these cases, the rank in modulo~$3$ spectral sequence is exactly the free rank plus the torsion rank in the integral case.
				Hence it remains to consider the kernel of the $d^2$ differential in the cases when the integral kernel is not the entire domain.
				By the rank nullity theorem, the rank of the image plus the nullity, the dimension of the kernel, is the dimension of the domain.
				
				When considering the spectral sequence modulo $3$, the rank of any differential is the same as in the integral case
				when the quotient of the preceding kernel by the image contains no torsion.
				When integral $3$-torsion exists, there are generators of the image which are $3$ times generators of the kernel.
				Therefore in the modulo $3$ spectral sequence these generators are now generators of the kernel.
				Hence in the modulo $3$ spectral sequence the rank is reduced by the integral torsion rank and the nullity increased by the same number.

				Since the modulo $3$ spectral sequence has coefficients in a field, there are no extension problems.
				As the total degree of the $d^2$ differential is $-1$ and $E_3=E_\infty$, $\dim(H^i(SU(3)/T^2;\mathbb{Z}_3))$
				is the sum of the ranks of total degree $i$ in the integral $E_3$-page plus the sum of the torsion ranks in total degrees $i$ and $i+1$.
				Hence, the modulo $3$ cohomology algebra is only consistent with the case when all torsion on the $E_\infty$-page of the spectral sequence
				is contained in the integral cohomology module.
				Therefore all additive extension problems are resolved and all the torsion elements in the spectral sequence are present in the integral cohomology.
				
				Now that we have deduced that the module structure of $H^*(\Lambda(SU(3)/T^2);\mathbb{Z})$
				is isomorphic to $A/I$,
				we have that $A/I$ is an associated graded algebra of $H^*(\Lambda(SU(3)/T^2);\mathbb{Z})$
				with respect to multiplication length filtration.
				Next we study the multiplicative extension problem.
				
				Multiplication on generators $\gamma_1,\gamma_2$ is in the image of the induced map of
				the evaluation fibration $H^*(SU(3)/T^2;\mathbb{Z})\to H^*(\Lambda(SU(3)/T^2);\mathbb{Z})$,
				hence the sub-algebra they generate contains no extension problems.
				Multiplication on generators $y_1,y_2$ is a free graded commutative sub-algebra, 
				so contains no extension problems.
				The algebra generated by $(x_4)_m$ is up to scalar multiplication a polynomial algebra generated by $y_1,y_2x_4$,
				hence also contains no extension problems
				
				Suppose the product $y_i\gamma_j$ for $1\leq i,j\leq 2$ in $A/I$
				contains additional summands in $H^*(\Lambda(SU(3)/T^2);\mathbb{Z})$.
				For dimensional reasons any summand must have the form $y_{i'}\gamma_{j'}$ for $1\leq i',j' \leq 2$. 
				However since $\gamma_i,y_i \in H^*(\Lambda(SU(3)/T^2);\mathbb{Z})$
				all possible additional summands have the same multiplication length showing that there is no multiplicative extension problem for $y_i\gamma_j$, $1\leq i,j\leq 2$ .
				
			\end{proof}

	\subsection{Free loop cohomology of $Sp(2)/T^2$}	
	
		\begin{theorem}\label{theorem:FreeLoopSp(2)/T}
				The integral algebra structure of the $E_\infty$-page of the Leray-Serre spectral sequence associated to the evaluation fibration of $Sp(2)/T^2$
				is $A/I$, where
			\begin{align*}
				A=\Lambda_\mathbb{Z}(&((x_6)_b\gamma_i,\;
				(x_6)_by_i,\;
				(x_2)_m(x_6)_by_1y_2,\;
				(x_2)_m(x_6)_b(y_1\gamma_2-y_2\gamma_1), \\
				& (x_2)_m(x_6)_by_2\gamma_1^3,\;
				(x_2)_m(x_6)_b(y_1\gamma_1+y_2\gamma_2),\;
			    (x_2)_m(x_6)_b(\gamma_1^2+\gamma_2^2),\\
			    & (x_2)_m(x_6)_b\gamma_1^4,\;
				(x_2)_m(x_6)_b\gamma_1^3\gamma_2)
			\end{align*}
			and
			\begin{align*}
				I=\langle&((x_2)_1^m-m!(x_2)_m)j,\;((x_6)_1^m-m!(x_6)_m)k, \\
				& (x_2)_a(x_6)_b(\gamma_1^2+\gamma_2^2),\;(x_2)_a(x_6)_b\gamma_1^4,\;
				2(x_2)_a(x_6)_b(y_1\gamma_1+y_2\gamma_2),\;(x_6)_b4y_1\gamma_1^3)\rangle
			\end{align*}
			where $i=1,2$, $m\geq 1$, $a,b\geq 1$, $|\gamma_i|=2$, $|(x_2)_m|=2m$, $|(x_6)_m|=6m$, $|y_i|=1$, 
			\begin{align*}
			    j\in\{& (x_6)_by_1y_2,\; (x_6)_by_1\gamma_2-y_2\gamma_1,\; (x_6)_by_2\gamma_1^3,\\ & (x_6)_by_1\gamma_1+y_2\gamma_2,\; (x_6)_b\gamma_1^3\gamma_2,\; (x_6)_b\gamma_1^2+\gamma_2^2,\; (x_6)_b\gamma_1^4 \}
			\end{align*}
			and
			\begin{align*}
			    k\in\{& \gamma_i,\; y_i,\; (x_2)_my_1y_2,\;(x_2)_m(y_1\gamma_2-y_2\gamma_1),\;(x_2)_my_2\gamma_1^3, \\
					&\;(x_2)_m(y_1\gamma_1+y_2\gamma_2),\;(x_2)_m\gamma_1^3\gamma_2),\;
					(x_2)_m(\gamma_1^2+\gamma_2^2),\;(x_2)_m\gamma_1^4)\}
			\end{align*}
			for $1\leq i \leq 2$.
		\end{theorem}
		
		\begin{proof}
			We consider the cohomology Leray-Serre spectral sequence $\{ E_r,d^r \}$ associated to the evaluation fibration of $Sp(2)/T^2$,
			\begin{equation*}
				\Omega(Sp(2)/T^2)\to\Lambda(Sp(2)/T^2)\to Sp(2)/T^2.
			\end{equation*}
			By (\ref{eq:H*Sp/T}),
			the cohomology of the base space $Sp(2)/T^2$ is
			\begin{equation*}
				H^*(Sp(2)/T^2;\mathbb{Z})=\frac{\mathbb{Z}[\gamma_1,\gamma_2]}{\langle\gamma_1^2+\gamma_2^2,\;\gamma_1^4\rangle}.
			\end{equation*}
			From (\ref{eq:spLoppCoHomology}), the cohomology of the fibre $\Omega(Sp(2)/T^2)$ is
			\begin{equation*}
				H^*(\Omega(Sp(2)/T^2);\mathbb{Z})=\Lambda_{\mathbb{Z}}(y_1,y_2)\otimes \Gamma_{\mathbb{Z}}[x_2,x_6]
			\end{equation*}
			where $|y_1|=1=|y_2|$, $|x_2|=2$, $|x_6|=6$.
			
			The elements on the $E_2$-page of the spectral sequence are generated additively by representative elements of the form
			\begin{equation*}
				(x_2)_a(x_6)_bP, \;\; (x_2)_a(x_6)_by_iP, \;\; (x_2)_a(x_6)_by_1y_2P
			\end{equation*}
			where $0\leq a,b$, $1\leq i \leq 2$ and $P\in \mathbb{Z}[\gamma_1,\gamma_2]$ is a monomial of degree at most $4$.
			
			By Theorem~\ref{theorem:SpDiff}, the only non-zero differentials in
			$\{ E_r,d^r \}$ are $d^2$ and $d^6$,
			which are none-zero only on generators $x_2$ and $x_6$ respectively.
			Hence the spectral sequence converges at the seventh page. 
			The differentials up to sign are given by
			\begin{equation}\label{eq:Sp(2)diff}
				d^2(x_2)=2(y_1\gamma_1+y_2\gamma_2), \;\;\;
				d^4(x_6)=2[y_1\gamma_1\gamma_2^2+y_2\gamma_2\gamma_1^2].
			\end{equation}
		 By using a different representative of $d^4(x_6)$,
			\begin{align}\label{eq:x6Non-trival}
				d^4([x_6])+d^2([x_2])\gamma_1^2 \nonumber
				&=2[y_1\gamma_1\gamma_2^2+y_2\gamma_1^2\gamma_2]+2[y_1\gamma_1^3+y_2\gamma_1^2\gamma_2] \\
				&=2[y_1\gamma_1^3-y_2\gamma_1^2\gamma_2]+2[y_1\gamma_1^3+y_2\gamma_1^2\gamma_2] \nonumber \\
				&=4[y_1\gamma_1^3]
			\end{align} 
			 we can see readily that it is a multiple of four.
			From now on we take $4[y_1\gamma_1^3]$ as the image of $d^4$.
			
			The monomial generators $\gamma_i$, $y_i$ and $(x_2)_m(x_6)_by_1y_2$ occur in $E_2^{*,0}$ or $E_2^{0,*}$ and are always in the kernel of the differentials,
			so they are algebra generators of the $E_\infty$-page.
			All relations on the $E^7$-page coming from the relations in 
				$H^*(\Omega(Sp(2)/T^2);\mathbb{Z}): 
                    (x_2)_1^m-m!(x_2)_m$ and $(x_6)_1^m-m!(x_6)_m$,
               the relations in $H^*(Sp(2)/T^2;\mathbb{Z}):
                    \gamma_1^2+\gamma_2^2$ and $\gamma_1^4$
               or are either in the image of $d^2: 2(y_1\gamma_1+y_2\gamma_2)$ or in the image of $d^6: 4y_1\gamma_1^3$
                hold on the $E_\infty$-page and therefore are in $I$ if they are in $A$.
			For this reason, since $E^2$ generators $(x_2)_m$ and $(x_6)_m$ will not be in $A$, all generators of $I$ must be considered up to multiple of $(x_2)_a$ and $(x_6)_b$.
			In addition we add $((x_2)_1^m-m!(x_2)_m)j$ and $((x_6)_1^m-m!(x_6)_m)k$ 
			for $j,k$ such that $(x_2)_mj$ and $(x_6)_mk$ are generators of the algebra.
			It remains to determine all generators of $A$.
			
			Elements on the $E_2$-page of the form $(x_2)_a(x_6)_by_1y_2P$ are in the kernel of $d^2$ and generated by $y_1y_2(x_2)_m(x_6)_b, (x_6)_by_i$ and $(x_6)_b\gamma_i$.
				
			Let $\phi\colon\mathbb{Z}[(x_6)_m,(x_2)_m,\gamma_1,\gamma_2] \to \mathbb{Z}[(x_6)_m,(x_2)_m,y_1,y_2,\gamma_1,\gamma_2]$ be the map defined by $d^2$ so that the following diagram commutes
			\begin{equation*}
				    \xymatrix{
							{\ker(\phi)} \ar[r]^(.5){} \ar[d]^(.45){q} & {\mathbb{Z}[(x_6)_m,(x_2)_m,\gamma_1,\gamma_2]} \ar[r]^(0.45){\phi} \ar[d]^(.45){q}  & {\mathbb{Z}[(x_6)_m,(x_2)_m,y_1,y_2,\gamma_1,\gamma_2]} \ar[d]^(.45){q} \\
							{\ker(d^2)} \ar[r]^(.45){} & {E_2} \ar[r]^-(.5){d^2} & {E_2}}
			\end{equation*}
			where $q$ is the quotient map by symmetric polynomials and divided polynomial relations
			\begin{equation*}
				\gamma_1^2+\gamma_2^2,\;  \gamma_1^4,\; (x_2)_1^m-m!(x_2)_m \text{ and } (x_6)_1^m-m!(x_6)_m.
			\end{equation*}
			By direct calculation, it can be seen that the kernel of $\phi$ is generated by $\gamma_1$,  $\gamma_2$ and $(x_6)_m$ which are also in the kernel of $d^2$.
			Thus the remaining elements of the kennel of $d^2$ of the form $(x_2)_m(x_6)_bP$ are elements of the ideal
			\begin{equation}\label{eq:IdealIntyiSp}
				q\phi^{-1}(\langle 2(y_1\gamma_1+y_2\gamma_2) \rangle \cap\langle\gamma_2^2+\gamma_1^2,\; \gamma_1^4 \rangle)
			\end{equation}
			where $q\phi^{-1}(2(y_1\gamma_1+y_2\gamma_2))$ spans $\text{Im}(\phi)$.
			Intersection (\ref{eq:IdealIntyiSp}) can be computed by Gr\"{o}bner basis. A basis with respect to lexicographic monomial ordering $y_2>y_1>\gamma_1>\gamma_1$ restricted to terms containing only a multiple of a single $y_i$ is given by
			\begin{equation}\label{eq:IdealInt1Sp}
			    2y_2\gamma_2\gamma_1^2+y_2\gamma_2^3+2y_1\gamma_2^2\gamma_1+2y_1\gamma_1^3,\;
			    2y_2\gamma_2\gamma_1^4+2y_1\gamma_1^5,\;
			    2y_2\gamma_2^2\gamma_1^3+2y_1\gamma_2^2\gamma_1^4.
			\end{equation}
			The $q\psi^{-1}$ image of (\ref{eq:IdealInt1Sp}) will, up to multiple of $(x_2)_m$ and $(x_6)_m$, generate the kernel of $d^2$ with elements of the form $(x_2)_m(x_4)_bP$.
			The $q\psi^{-1}$ image of $2y_2\gamma_2\gamma_1^2+y_2\gamma_2^3+2y_1\gamma_2^2\gamma_1+2y_1\gamma_1^3$
			is trivial.
			However since it is not yet contained in the generators, we add relation
			\begin{equation*}
			(x_2)_m(x_4)_b(\gamma_1^2+\gamma_2^2)
			\end{equation*}
			as a generator in order to minimize the number of algebra generators required.				
			The $q\psi^{-1}$ image of $2y_2\gamma_2^2\gamma_1^3+2y_1\gamma_2^2\gamma_1^4$
			is $\gamma_2\gamma_1^3$.
			Hence, we take 
			\begin{equation*}
				(x_2)_m,(x_6)_b\gamma_2\gamma_1^3
			\end{equation*}
			as generators of the algebra.
			The $q\psi^{-1}$ image of $2y_2\gamma_2\gamma_1^4+2y_1\gamma_1^5$
			is trivial.
			Since it is not yet contained in the generators, we add relation
			\begin{equation*}
			(x_2)_m(x_4)_b\gamma_1^4
			\end{equation*}
			as a generator in order to minimize the number of algebra generators required.
				
			Let $\psi\colon\mathbb{Z}[(x_6)_b,(x_2)_my_1, (x_2)_my_2,\gamma_1,\gamma_2] \to \mathbb{Z}[(x_2)_m,(x_6)_b,y_1,y_2,\gamma_1,\gamma_2]$ 
			be the map defined by $d^2$ so that the following diagram commutes
			\begin{equation*}
				\xymatrix{
							{\ker(\psi)} \ar[r]^(.5){} \ar[d]^(.45){q} & {\mathbb{Z}[(x_6)_b,(x_2)_my_1, (x_2)_my_2,\gamma_1,\gamma_2]} \ar[r]^(0.525){\psi} \ar[d]^(.45){q}  & {\mathbb{Z}[(x_2)_m,(x_6)_b,y_1,y_2,\gamma_1,\gamma_2]} \ar[d]^(.45){q} \\
							{\ker(d^2)} \ar[r]^(.45){} & {E_2} \ar[r]^-(.5){d^2} & {E_2}}
			\end{equation*}
			where $q$ is the quotient map by relations
			\begin{equation*}
				\gamma_1^2+\gamma_2^2,\; \gamma_1^4,\; (x_2)_1^m-m!(x_2)_m \text{ and } (x_6)_1^m-m!(x_6)_m.
			\end{equation*}
			By direct calculation, it can be seen that the kernel of $\psi$ is generated by
			$\gamma_1$, $\gamma_2$, $(x_6)_m$ and $(x_2)_m(y_1\gamma_1+y_2\gamma_2)$ the image of $d^2$.
		
			Thus the remaining element of the kernel of $d^2$ of the form $(x_2)_m(x_6)_by_iP$ are elements of the ideal
			\begin{equation}\label{eq:IdealIntSp}
				q\psi^{-1}(\langle 2y_1y_2\gamma_1,\; y_1y_2\gamma_2\rangle \cap \langle \gamma_2^2+\gamma_1^2,\; \gamma_1^4 \rangle)
			\end{equation}
			where $q\psi^{-1}(y_1y_2(2\gamma_2+\gamma_1),\;y_1y_2(\gamma_2+2\gamma_1))$ spans $\text{Im}(\psi)$.
				
			It is straightforward to see that intersection (\ref{eq:IdealIntSp}) is the ideal
			\begin{equation}\label{eq:IdealyiSp}
				\langle 2y_1y_2(\gamma_1^2+\gamma_2^2),\; 2y_1y_2\gamma_1^4 \rangle.
			\end{equation}
			The generators of the $q\psi^{-1}$ image of (\ref{eq:IdealIntSp}) will up to multiple of $(x_2)_m$ and $(x_6)_m$, generate the kernel of $d^2$ with elements of the form $(x_2)_m(x_4)_by_iP$.
			The $q\psi^{-1}$ image of generator $2y_1y_2(\gamma_1^2+\gamma_2^2)$ of (\ref{eq:IdealyiSp}) is $[y_1\gamma_1-y_2\gamma_2]$. 
			Hence we take both 
			\begin{equation*}
				(x_2)_m \text{ and }(x_6)_b(y_1\gamma_1-y_2\gamma_2)
			\end{equation*}
			as generators of the algebra.
			The $q\psi^{-1}$ image of generator $y_1y_2\gamma_1^4$ of (\ref{eq:IdealyiSp}) is $[y_2\gamma_1^3]$.
			Hence we take 
			\begin{equation*}
				(x_2)_m,(x_6)_by_2\gamma_1^3
			\end{equation*}
			as generators of the algebra.
			It remains to consider when the $d^6$ differential is non-trivial.
			
			Notice that by (\ref{eq:x6Non-trival}), on the $E^6$-page 
			\begin{align*}
			    & d^6([(x_6)_my_1])=4(x_6)_{m-1}[y_1^2\gamma_1^3\gamma_2]=0 \\
				& d^6([(x_6)_my_2])=4(x_6)_{m-1}[y_1y_2\gamma_1^4]=0 \\
				& d^6([(x_6)_m\gamma_1])=4(x_6)_{m-1}[y_1\gamma_1^4]=0 \\
				\text{and} \;\;  & d^6([(x_6)_m\gamma_2])=4(x_6)_{m-1}[y_1\gamma_1^3\gamma_2]=0.
			\end{align*}
			Therefore the only classes on the $E^6$ page not in the kernel of $d^6$ are $[(x_6)_m]$.
		\end{proof}
			
		\begin{theorem}
			 The cohomology $H^*(\Lambda(Sp(2)/T^2);\mathbb{Z})$ is isomorphic as a module to $A/I$ given in Theorem~\ref{theorem:FreeLoopSp(2)/T} up to order of $2$-torsion.
		\end{theorem}
			 
		\begin{proof}
			As in the proof of Theorem~\ref{corollary:L(SU(3)/T^2)} by considering a Gr\"obner basis of elemts in $I$
			it can be seen that torsion on the $E_\infty$-page of $\{ E_r,d^r \}$ is a power of $2$, at most $4$.
			Hence we consider the spectral sequence $\{ E_r,d^r \}$ over the field of characteristic $2$.
			Since the only non-zero differentials $d^2$ and $d^6$ have bidegree $(2,-1)$ and $(6,-5)$ respectively, 
			for exactly the same reasons as for the modulo $3$ spectral sequence in Theorem~\ref{theorem:L(SU(3)/T^2)},
			all torsion on the $E_\infty$-page survives the additive extension problem over $\mathbb{Z}$.
			The only remaining additive extension problem is weather the $4$-torsion generated by
			$[(x_2)_ay_1\gamma^3_1]$ on the $E_\infty$,
			is $2$-torsion or $4$-torsion in $H^*(\Lambda (Sp(2)/T^2);\mathbb{Z})$. 
		\end{proof}

	\subsection{Free loop cohomology of $G_2/T^2$}
		
		\begin{theorem}\label{theorem:L(G2/T^2)}
			The integral algebra structure of the $E_\infty$-page of the Leray-Serre spectral sequence associated to the evaluation fibration of $G_2/T^2$
			is $A/I$, where
			\begin{align*}
				A=\Lambda_{\mathbb{Z}}(&(b_{10})_l\gamma_i,\; (b_{10})_lt_3,\; (b_{10})_ly_i,\; (a_2)_m(b_{10})_ly_1y_2,\\
				&(a_2)_m(b_{10})_l(y_1(\gamma_1+\gamma_2)+y_2\gamma_2),\;
				(a_2)_m(b_{10})_l(t_3\gamma_1^2(y_1-2y_2)),\\
				&(a_2)_m(b_{10})_l(y_1(2\gamma_1+\gamma_2)+y_2(\gamma_1+2\gamma_2)),\;
				(a_2)_m(b_{10})_l(\gamma_1^2+\gamma_1\gamma_2+\gamma_2^2)\\
				&(a_2)_m(b_{10})_l(2t_3-\gamma_1^3),\;
				(a_2)_m(b_{10})_lt_3,\;
				(a_2)_m(b_{10})_l\gamma_1^3,\;
				(a_2)_m(b_{10})_l\gamma_1^2\gamma_2
				)
			\end{align*}
			and
			\begin{align*}
				I=\langle &(a_2^{m}-m!/2^{\lfloor\frac{m}{2}\rfloor}(a_2)_m)j,\;
				(b_{10}^m-m!(b_{10})_m)k,\;
				(a_2)_h(b_{10})_l(2t_3-\gamma_1^3), \\
				& (a_2)_h(b_{10})_l(\gamma_1^2+\gamma_1\gamma_2+\gamma_2^2),\;
				(a_2)_h(b_{10})_lt_3^2,\;
				2(a_2)_{s+1}(b_{10})_l(y_1(2\gamma_1+\gamma_2)+y_2(\gamma_1+2\gamma_2)),\\ & (a_2)_{s}(b_{10})_l(y_1(2\gamma_1+\gamma_2)+y_2(\gamma_1+2\gamma_2)),\;  
				3(a_2)_h(b_{10})_ly_1t_3\gamma_1^2 \rangle
			\end{align*}
			where $s,l,h\geq 0$ with $s$ even, $m\geq 1$, $(a_2)_0=1$,
			$|\gamma_i|=2$, $|y_i|=1$, $|(a_2)_m|=2m$, $|(b_{10})_m|=10m$, 
			\begin{align*}
			    j\in\{& (b_{10})_ly_1y_2,\; (b_{10})_l(y_1(\gamma_1+\gamma_2)+y_2\gamma_2),\\ 
			    & (b_{10})_l(t_3\gamma_1^2(y_1-2y_2)),\;
			    (b_{10})_l\gamma_1^2\gamma_2,\; (b_{10})_l(y_1(2\gamma_1+\gamma_2)+y_2(\gamma+2\gamma_2)),\\
			    &(b_{10})_l(\gamma_1^2+\gamma_1\gamma_2+\gamma_2^2),\;
				(b_{10})_l(2t_3-\gamma_1^3),\;
				(b_{10})_lt_3,\;
				(b_{10})_l\gamma_1^3,\;
			    (b_{10})_l\gamma_1^2\gamma_2
				\}
			\end{align*}
			and
			\begin{align*}
			    k\in\{& \gamma_i,\; y_i,\; t_3,\; (a_2)_sy_1y_2,\; 
			    (a_2)_l(y_1(\gamma_1+\gamma_2)+y_2\gamma_2),\; \\
				&(a_2)_l(t_3\gamma_1^2(y_1-2y_2)),\;
				(a_2)_l\gamma_1^2\gamma_2,\; (a_2)_m(y_1(2\gamma_1+\gamma_2)+y_2(\gamma+2\gamma_2)),\\
				&(a_2)_l(\gamma_1^2+\gamma_1\gamma_2+\gamma_2^2),\;
				(a_2)_l(2t_3-\gamma_1^3),\;
				(a_2)_lt_3,\;
				(a_2)_l\gamma_1^3,\;
				(a_2)_l\gamma_1^2\gamma_2
				\}
			\end{align*}
			for $1\leq i \leq 2$.
		\end{theorem}	
		
		\begin{proof}
			We consider the cohomology Leray-Serre spectral sequence $\{ E_r,d^r \}$ associated to the evaluation fibration of $G_2/T^2$,
			\begin{equation*}
				\Omega(G_2/T^2)\to\Lambda(G_2/T^2)\to G_2/T^2.
			\end{equation*}
			The cohomology of $G_2/T^2$ is given by
			\begin{equation*}
			    \frac{\mathbb{Z}[\gamma_1,\gamma_2,t_3]}
			    {\langle \gamma_1^2+\gamma_1\gamma_2+\gamma_2^2, 2t_3-\gamma_1^3,t_3^2 \rangle}
			\end{equation*}
			where $|\gamma_1|=|\gamma_2|=2$ and $|t_3|=6$.
			The cohomology of $\Omega(G_2/T^2)$ is 
			\begin{equation*}
				\Lambda(y_1,y_2)\otimes \frac{\mathbb{Z}[(a_2)_1,(a_2)_2,\dots]}
				{\langle a_2^{m}-(m!/2^{\lfloor\frac{m}{2}\rfloor})(a_2)_{m}\rangle}
				\otimes \Gamma_{\mathbb{Z}}[b_{10}]
			\end{equation*}
			where $|y_1|=|y_2|=1$, $|a_2|=2$ and $|b_{10}|=10$.
			
			The $E_2$-page of the spectral sequence is generated additively by elements of the form
			\begin{equation*}
				(a_2)_h(b_{10})_lP, \;\; (a_2)_h(b_{10})_ly_iP, \;\; (a_2)_h(b_{10})_ly_1y_2P
			\end{equation*}
			where $0\leq l,h$, $1\leq i \leq 2$ and $P\in \mathbb{Z}[\gamma_1,\gamma_2, t_3]$
			is a monomial of degree at most $6$,
		    taking $t_3$ as monomial of degree~$3$. 
			By Theorem~\ref{theorem:DiffG2}, the only non-zero differentials are $d^2$ and $d^{10}$,
			which are non-zero only on generators $a_2$ and $b_{10}$, respectively.	
			The differentials up to sign are given by
			\begin{equation}\label{eq:G2differetials}
				d^2([a_2])=[y_1(\gamma_1+2\gamma_2)+y_2(2\gamma_1+\gamma_2)], \;\;\;
				d^{10}([b_{10}])=[3y_1t_3\gamma_1^2].
			\end{equation}	
			In particular the spectral sequence converges by the $E_{11}$-page.
			The monomial generators $\gamma_i$, $t_3$, $y_i$ and $(a_2)_s(b_{10})_ly_1y_2$ occur in $E_2^{*,0}$ or $E_2^{0,*}$ and are always in the kernel of the differentials,
			so are algebra generators of the $E_\infty$-page.
			All relations on the $E_{11}$-page coming from the relations in 
				$H^*(\Omega(G_2/T^2);\mathbb{Z}):
                    a_2^{m}-m!/2^{\lfloor\frac{m}{2}\rfloor}(a_2)_m$ and 
                    $(b_{10})_1^m-m!(b_{10})_m$, the relations in $H^*(G_2/T^2;\mathbb{Z}):
                    \gamma_1^2+\gamma_1\gamma_2+\gamma_2^2,
                    2t_3-\gamma_1^3$ and
                    $t_3^2$, or are in the image of $d^2: y_1(\gamma_1+2\gamma_2)+y_2(2\gamma_1+\gamma_2)$ and in the image of $d^{10}: 3y_1t_3\gamma_1^2$
			 hold on the $E_\infty$-page and therefore are in $I$ if they are in $A$.
			For this reason, since $E_2$ generators $(a_2)_m$ and $(b_{10})_m$ will not be in $A$ all generators of $I$, must be considered up to multiple of $(a_2)_h$ and $(b_{10})_l$.
			However due to the non-divided polynomial multiplicative structure on $\langle (a_2)_m \rangle$, unlike for previous flag manifolds for $s\geq 1$ even 
			\begin{equation*}
				d^2([(a_2)_s])=[2(a_2)_{s-1}(y_1(\gamma_1+2\gamma_2)+y_2(2\gamma_1+\gamma_2))].
			\end{equation*}
			Hence, we add generators
			\begin{equation*}
			    (a_2)_{s}(y_1(\gamma_1+2\gamma_2)+y_2(2\gamma_1+\gamma_2))
			    \text{ and }
			    2(a_2)_{s+1}(y_1(\gamma_1+2\gamma_2)+y_2(2\gamma_1+\gamma_2))
			\end{equation*}
			to $I$ instead of $(a_2)_m(y_1(\gamma_1+2\gamma_2)+y_2(2\gamma_1+\gamma_2))$.
			In addition, we add $(a_2^{m}-m!/2^{\lfloor\frac{m}{2}\rfloor}(a_2)_m)j$
			and $((b_{10})_1^m-m!(b_{10})_m)k$ 
			for $j,k$ such that $(a_2)_mj$ and $(b_{10})_mk$ are generators of the algebra.
			It remains to determine all generators of $A$.
			
			The image of the differential and symmetric ideal generators are similar to the case of the proof of Theorem~\ref{theorem:L(SU(3)/T^2)}.
			Hence the argument is similar to the proof of Theorem~\ref{theorem:L(SU(3)/T^2)}, we add
			\begin{align*}
			    &(a_2)_m(b_{10})_l(y_1(\gamma_1+\gamma_2)+y_2\gamma_2)\;
				(a_2)_m(b_{10})_l(t_3\gamma_1^2(y_1-2y_2)),\\
				&(a_2)_m(b_{10})_l(y_1(2\gamma_1+\gamma_2)+y_2(\gamma_1+2\gamma_2)),\;
				(a_2)_m(b_{10})_l(\gamma_1^2+\gamma_1\gamma_2+\gamma_2^2)\\
				&(a_2)_m(b_{10})_l(2t_3-\gamma_1^3),\;
				(a_2)_m(b_{10})_lt_3,\;
				(a_2)_m(b_{10})_l\gamma_1^3
				\text{ and }(a_2)_m(b_{10})_l\gamma_1^2\gamma_2
			\end{align*}
			as gnerators of $A$. 
			It remains to consider when the $d^{10}$ differential is non-trivial.
			
			Using (\ref{eq:G2differetials}), it can be seen that the images of $[b_{10}\gamma_i]$ and $[b_{10}y_i]$
			are trivial.
			Therefore, the only classes on the $E^{10}$ page not in the kernel of $d^{10}$ are $[(b_{10})_m]$.
		\end{proof}
			
		\begin{theorem}
			 The cohomology $H^*(\Lambda(G_2/T^2);\mathbb{Z})$ is isomorphic as an algebra to $A/I$ given in Theorem~\ref{theorem:L(G2/T^2)} up to order of $2$-torsion.   
		\end{theorem}
			 
		\begin{proof}
			As in the proof of Theorem~\ref{corollary:L(SU(3)/T^2)} by considering a Gr\"obner basis of elements in $I$
			all torsion on the $E_\infty$ is of rank $2$, $3$, $6$ or $12$.
			Most module extension problems is resolved in the same way as previous cases by considering the module $2$ and modulo $3$ spectral sequences. However this does not determine torsion
			of rank $12$.
		\end{proof}

\bibliographystyle{amsplain}
\bibliography{Ref}

\end{document}